\documentclass[11pt,reqno]{amsart}

\usepackage{version}
\includeversion{arxiv}
\excludeversion{journal}

\usepackage[utf8]{inputenc}
\usepackage[T1]{fontenc}
\usepackage[english]{babel} 
\usepackage{amssymb}
\usepackage{stmaryrd}
\usepackage{mathrsfs} 
\usepackage{graphicx}  
\usepackage{paralist}
\usepackage{booktabs,multicol,array,float}
\usepackage[colorlinks=true]{hyperref} 
\usepackage{csquotes}

\usepackage{tikz}
\usetikzlibrary{cd,arrows}

\usepackage{listings}
\usepackage[backend=biber,style=alphabetic,maxbibnames=4,doi=false,isbn=false,url=false]{biblatex}
\bibliography{euluffi}

\theoremstyle{plain}
\newtheorem{thm}{Theorem}[section]
\newtheorem{prp}[thm]{Proposition}
\newtheorem{cor}[thm]{Corollary}
\newtheorem{lem}[thm]{Lemma}
\newtheorem{cnj}[thm]{Conjecture}
\theoremstyle{definition}
\newtheorem{dfn}[thm]{Definition}
\newtheorem{ntn}[thm]{Notation}
\newtheorem*{cnv}{Convention}
\theoremstyle{remark}
\newtheorem{rmk}[thm]{Remark}
\newtheorem{exa}[thm]{Example}

\renewcommand{\AA}{\mathbb{A}}
\newcommand{\CC}{\mathbb{C}}
\newcommand{\FF}{\mathbb{F}}
\newcommand{\GG}{\mathbb{G}}

\newcommand{\ZZ}{\mathbb{Z}}
\newcommand{\NN}{\mathbb{N}}
\newcommand{\PP}{\mathbb{P}}
\newcommand{\KK}{\mathbb{K}}
\newcommand{\LL}{\mathbb{L}}
\newcommand{\TT}{\mathbb{T}}

\newcommand{\B}{\mathcal{B}}
\newcommand{\C}{\mathcal{C}}
\newcommand{\E}{\mathcal{E}}
\newcommand{\F}{\mathcal{F}}
\newcommand{\I}{\mathcal{I}}
\newcommand{\K}{\mathcal{K}}
\renewcommand{\L}{\mathcal{L}}
\newcommand{\M}{\mathsf{M}}
\newcommand{\T}{\mathcal{T}}
\newcommand{\U}{\mathcal{U}}
\newcommand{\V}{\mathcal{V}}

\newcommand{\ol}[1]{\overline{#1}}
\newcommand{\wh}[1]{\widehat{#1}}
\newcommand{\wt}[1]{\widetilde{#1}}

\newcommand{\onto}{\twoheadrightarrow}

\newcommand{\olstar}{\mathbin{\ol{\star}}}

\newcommand{\set}[1]{{\left\{#1\right\}}}
\newcommand{\abs}[1]{{\left\vert#1\right\vert}}
\newcommand{\ideal}[1]{{\left\langle#1\right\rangle}}

\DeclareMathOperator{\cl}{cl}

\DeclareMathOperator{\nullity}{null}

\DeclareMathOperator{\rk}{rk}
\DeclareMathOperator{\Spec}{Spec}
\DeclareMathOperator{\Sym}{Sym}
\DeclareMathOperator{\Var}{Var}

\title[On a conjecture of Aluffi]{Graph hypersurfaces with torus action\\ and a conjecture of Aluffi}

\author[G.~Denham]{Graham Denham}
\address{\linebreak
Graham Denham\\
Department of Mathematics, University of Western Ontario\\ 
London, Ontario, Canada N6A 5B7
}
\email{\href{gdenham@uwo.ca}{gdenham@uwo.ca}}

\author[D.~Pol]{Delphine Pol}
\address{ \linebreak
Delphine Pol\\
Department of Mathematics, TU Kaiserslautern\\
67663 Kaiserslautern\\
Germany
}
\email{\href{pol@mathematik.uni-kl.de}{pol@mathematik.uni-kl.de}}

\author[M.~Schulze]{Mathias Schulze}
\address{ \linebreak
Mathias Schulze\\
Department of Mathematics, TU Kaiserslautern\\
67663 Kaiserslautern\\
Germany
}
\email{\href{mschulze@mathematik.uni-kl.de}{mschulze@mathematik.uni-kl.de}}
 
\author[U.~Walther]{Uli Walther}
\address{\linebreak 
Uli Walther\\
Department of Mathematics, Purdue University\\
West Lafayette, IN 47907, USA
}
\email{\href{walther@math.purdue.edu}{walther@math.purdue.edu}}
 
\subjclass[2010]{Primary 05C31; Secondary 13D15, 14M12, 14N20, 14R20, 81Q30}

\keywords{Configuration, matroid, star graph, Euler characteristic, Grothendieck ring, torus action, Feynman, Kirchhoff, Symanzik}

\thanks{GD supported by NSERC of Canada. 
DP supported by a Humboldt Research Fellowship for Postdoctoral Researchers.
UW supported in part by the National Science Foundation under grant 2100288, and by a Simons Foundation Collaboration Grant for Mathematicians.}

\numberwithin{equation}{section} 

\begin{document}

\begin{abstract}
Generalizing the $\star$-graphs of M\"uller-Stach and Westrich, we describe a class of graphs whose associated graph hypersurface is equipped with a non-trivial torus action.
For such graphs, we show that the Euler characteristic of the corresponding projective graph hypersurface complement is zero.
In contrast, we also show that the Euler characteristic in question can take any integer value for a suitable graph.
This disproves a conjecture of Aluffi in a strong sense.
\end{abstract}

\maketitle
\tableofcontents

\section{Introduction}

Let $G=(\V,\E)$ be an undirected graph on the vertex set $\V$ with edge set $\E$. 
Classically one associates to it the \emph{Kirchhoff polynomial} $\psi_G$, the sum of weights of all spanning trees, where the weight of a tree is the product of all its edge weights, considered as formal variables. 
In the last two decades, the \emph{graph hypersurfaces} defined by these polynomials have attracted considerable attention in the literature, largely because they appear in integrands for Feynman integrals (see \cite{Alu14,BBKP19,BS12,BSY14}).
Since graph hypersurfaces are in some sense fairly complex (see \cite{BB03}), even relatively coarse information is highly valued and not easy to obtain. 


By Kirchhoff's Matrix-Tree Theorem, $\psi_G$ appears as any cofactor of the weighted Laplacian of $G$ (provided $G$ is a connected graph).
A more general point of view was developed by Bloch, Esnault and Kreimer and further by Patterson (see \cite{BEK06,Pat10}): 
A submatrix of the weighted Laplacian obtained by deleting a row and corresponding column has a more intrinsic interpretation.
It is a matrix of the generic, diagonal bilinear form on $\KK^\E$ restricted to the subspace $W_G\subseteq\ZZ^\E$ of all incidence vectors of $G$.
As a consequence, $\psi_G$ arises as a determinant of this restricted bilinear form $Q_G$.

This motivates an analogous construction for an arbitrary linear subspace $W\subseteq\KK^\E$ for some field $\KK$, called a \emph{configuration} by the authors above.
It results in a \emph{configuration form} $Q_W$ whose matrix entries are  Hadamard products (see Remark~\ref{12}).
Its determinant $\psi_W$ is the \emph{configuration polynomial}. 
These polynomials are, from some points of view, more natural objects of study than the graph polynomials. 
In particular, the configuration point of view has recently led to new results on the singularities of graph hypersurfaces (see \cite{DSW21}).

\smallskip

In this paper we focus on the \emph{projective} graph hypersurface $X_G$ defined by $\psi_G$ in $\PP\KK^\E$, and its complement $Y_G$.
If $G$ consists entirely of loops\footnote{By a loop we mean a self-loop, an edge that connects a vertex to itself.}, then $\psi_G=1$ (see Remark~\ref{12}.\eqref{12a}). 
To avoid triviality, then, we adopt the following

\begin{cnv}
We assume that $G$ has at least one edge which is not a loop.
\end{cnv}

Our goal is to understand the Euler characteristic of the variety $Y_G$ (for $\KK=\CC$), and more generally the class $[Y_G]$ of $Y_G$ in the \emph{Grothendieck ring} $\K_0(\Var_\KK)$ of varieties over $\KK$, modulo the class $\TT:=[\GG_m]$ of the $1$-torus $\GG_m$. 
This investigation is complementary to the work of Belkale and Brosnan (see  \cite{BB03}) who studied the class of the \emph{affine} cone of $X_G$ in a localization of $\K_0(\Var_\KK)$ where $\TT$ is invertible.

\smallskip

For some basic families of graphs, a computation of $[Y_G]$ can be found in the literature:

\begin{asparaitem}

\item $[Y_G]=1$ for graphs on two vertices (see Remark \ref{53}),

\item $[Y_G]\equiv(-1)^{\abs{\E}-1}\mod\TT$ for cycle graphs (see \cite[Cor.~3.14]{AM09}), and

\item $[Y_G]\equiv 0\mod\TT$ for wheel graphs (using \cite[Prop.~49]{BS12}). 

\end{asparaitem}


In view of such computations, Aluffi made a conjecture on the Euler characteristic of $Y_G$ (see \cite[Conj.~3.6]{BM13}). 
We give a modified, dual formulation.

\begin{cnj}[Aluffi's Conjecture]\label{40}
The Euler characteristic of the complex graph hypersurface complement $Y_G$ has absolute value at most $1$, that is, 
\[
\chi(Y_G)\in\set{-1,0,1}.
\]
\end{cnj}


The original conjecture involves the \emph{Symanzik polynomial} of $G$ instead of $\psi_G$.
For planar graphs $G$ this agrees with the Kirchhoff polynomial $\psi_{G^\perp}$ of the dual graph $G^\perp$.
Our dual formulation of Aluffi's Conjecture thus coincides with the original one for planar graphs.

Let $Y_G^\circ$ denote the intersection of $Y_G$ with the standard open torus orbit of $\PP\KK^\E$.
If $G$ is planar, then $Y_G^\circ$ is identified with $Y_{G^\perp}^\circ$, via the (standard) Cremona transformation (see~\cite[Rem.~1.7]{BEK06}). 
This observation suggests that the torus hypersurface complements $Y^\circ_G$ should be primary objects of study, and also that one should make essential use of duality. 
It happens that the stratification of $Y_G$ by coordinate subspaces in $\PP\KK^\E$ interacts very pleasantly with both the graph structure and the Cremona transformations within the strata.
We make use of this to establish inclusion/exclusion formul\ae\ (see Proposition~\ref{56}), and we demonstrate their use for computing of $[Y_G]\mod\TT$. 

By M\"obius inversion, such formul\ae\ come in pairs of coupled triangular systems of equations, with unknowns $[Y_G]$ and $[Y_G^\circ]$.
The equality of leading terms $[Y_{\vphantom{G^\perp}\smash{G}}^\circ]=[Y_{G^\perp}^\circ]$ allows one to solve if in each step if either $[Y_G]$ or $[Y_{G^\perp}]$ is known. 
Here, we work in the more natural and general setting of complements $Y_W$ of \emph{configuration hypersurfaces} (see Definition~\ref{9}).

\smallskip

In order to solve the systems of equations that arise above, we identify graphs $G$ for which $X_G$ is an integral scheme and admits a non-trivial $\GG_m$-action with fixed point scheme $(X_G)^{\GG_m}$.
Then $[X_G]\equiv[(X_G)^{\GG_m}]\mod\TT$ by a result of Bia{\l}ynicki-Birula (see \cite{Bia73b}).
This approach was inspired by the work of M\"uller-Stach and Westrich (see \cite{MSW15}) who applied the Bia{\l}ynicki-Birula decomposition (see \cite{Bia73a}) to a non-singular model of $X_G$.

There are some trivial sources for such non-trivial torus actions, such as coloops and nexi (that is, cut-vertices) in $G$. 
It is also easy to see that deletion of loops and parallel edges leaves $[Y_G]\mod\TT$ unchanged (see Proposition~\ref{50}).
After such reductions, one is led to consider $2$-connected simple graphs $G$, which rules out the possibility of non-trivial monomial torus actions (see \cite[Prop.~3.8]{DSW21}). 
M\"uller-Stach and Westrich provide another source for non-trivial torus actions if $G^\perp$ is a so-called \emph{$\star$-graph}. 
This is a class of $2$-connected (planar) polygonal graphs (see Definition~\ref{29} and Remarks~\ref{25} and \ref{12}.\eqref{12e}). 
In their case, $G$ is a cone (see Proposition~\ref{16}) and the action is induced by conjugating the symmetric bilinear form $Q_G$ by a suitable diagonal action.

We significantly relax the hypotheses for such torus actions, by eliminating any condition on the dual graph, or indeed on planarity.
Our notion of a \emph{fat nexus} generalizes both the notions of apex and nexus (see Definition~\ref{0} and Remark~\ref{35}.\eqref{35e} and \eqref{35b}).
It is a vertex $v_0\in\V$ which admits a partition $\V=\set{v_0}\sqcup\V_1\sqcup\V_2$ such that each edge between $\V_1$ and $\V_2$ lies in the neighborhood $\V_0$ of $v_0$ (see Figure~\ref{110}).
Given a simple graph $G$ with fat nexus, we establish a non-monomial $\GG_m$-action on $X_G$ by conjugation of $Q_G$, identify the fixed point scheme and conclude that $[Y_G]\equiv0\mod\TT$ (see Theorem~\ref{10}).
This yields many examples of graphs supporting Aluffi's Conjecture~\ref{60} (see Corollary~\ref{42}).

At this point, the notion of a fat nexus with its accompanying torus action remains a graphical concept: we do not know how to lift it from graph hypersurfaces to general configuration hypersurfaces.


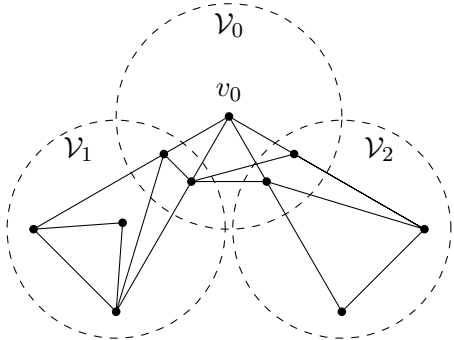
\begin{figure}[ht]
\begin{tikzpicture}[scale=0.5,baseline=(current bounding box.center)]
\tikzstyle{every node}=[circle,draw,inner sep=1pt,fill=black]
\tikzstyle{empty}=[draw=none,fill=none]
\draw (0,0) node [label=above:$v_0$] {};
\draw[dashed] (0,0) circle (3);
\draw[dashed] (-3,-3) circle (2.9);
\draw[dashed] (3,-3) circle (2.9);
\draw (0,1.8) node [empty,above] {$\V_0$};
\draw (-4,-1.5) node [empty,above] {$\V_1$};
\draw (4,-1.5) node [empty,above] {$\V_2$};
\draw \foreach \x in {1,2,4,5} {(0,0) -- (-30*\x:2) node {}};
\draw \foreach \x in {1,2,4,5} {(-30*\x:6) node {}};
\draw (-60:2) -- (-120:2) -- (-30:2);
\draw (-30:6) -- (-60:2) -- (-60:6) -- (-30:6) -- (-30:2) -- cycle;
\draw (-150:2) -- (-120:6) -- (-120:2) -- (-150:2) -- (-150:6) -- (-135:4) node {} -- (-120:6) -- (-150:6);
\end{tikzpicture}
\caption{A fat nexus $v_0$ with defining vertex partition.}\label{110}
\end{figure}

\smallskip

While our results on $[Y_G]\mod\TT$ for (co)loops and multiple edges in $G$ are proved rather directly, the one for edges in series relies on a more complicated argument using inclusion/exclusion and duality (see Corollary~\ref{59} and \cite[Prop.~5.2]{AM11}).
It leads to examples of planar graphs $G$ with edges in series for which $[Y_G]\mod\TT$ takes \emph{any} integer value (see Example~\ref{60}).
However, these graphs are physically not very relevant and this failure of Aluffi's Conjecture~\ref{40} seems to be somewhat artificial.

Applying our formul\ae\ (see Appendix~\ref{101}) to small graphs without fat nexi, we are able to compute $[Y_G]\equiv0\mod\TT$ for several new examples (see Appendix \ref{102}). 
As a particular result, we exhibit a planar simple graph without edges in series that violates Aluffi's Conjecture (see Example~\ref{86}). 
One can thus view the fat nexus property as a significant sign of lack of complexity of a graph.
The fat nexus hypothesis to our positive result on Aluffi's Conjecture is not just an artifact of the method of proof, but gives evidence of some serious obstructions to the conjecture.

\smallskip

The paper is organized as follows. 
In \S\ref{100} and Appendices~\ref{101} and \ref{102} we give an overview of our results.
In \S\ref{89} we show that our notion of fat nexus generalizes the $\star$-graphs of M\"uller-Stach and Westrich (see \cite{MSW15}).
In \S\ref{93} we review the basics on configurations, underlying matroids, configuration forms and configuration polynomials, generalizing Laplacians and Kirchhoff polynomials.
In \S\ref{91} we describe $[Y_G]\mod\TT$ for graphs $G$ with (co)loops, multiple edges, disconnection or nexi. 
In \S\ref{94} we explain how fat nexi lead to torus actions which allow us to show that $[Y_G]\equiv0\mod\TT$. 
In \S\ref{92} we establish formul\ae\ for $[Y_G]\mod\TT$ that arise from the toric stratification of $\PP\KK^\E$, M\"obius inversion and duality.
In \S\ref{95} and \S\ref{96} we compute $[Y_G]\mod\TT$ for certain wheel-like graphs with subdivided edges, and $[Y_W]\mod\TT$ if the underlying matroid is uniform of (co)rank $2$. 
\begin{arxiv}
Appendix~\ref{103} contains a Python implementation of our formul\ae\ which was used to verify our calculations.
\end{arxiv}

\subsection*{Acknowledgments}

We gratefully acknowledge support by the Bernoulli Center at EPFL during a \enquote{Bernoulli Brainstorm} in February 2019, and by the Centro de Giorgi in Pisa during a \enquote{Research in Pairs} in February 2020.
We thank Masahiko Yoshinaga for pointing out the paper~\cite{BM13} to the second author, and Erik Panzer for helpful discussions. 
We are grateful to the referees for a careful reading of the manuscript and resulting improvements to the exposition.

\section{Summary of results}\label{100}

Our positive result concerning Aluffi's Conjecture~\ref{40} involves the graph-theoretic notions of \emph{simplification}, \emph{vertex connectivity} and \emph{fat nexus}.
While the former two are standard, the latter is tailored to our problem.


\begin{dfn}[Simplification]
The \emph{simplification} $\wt G=(\wt\V,\wt\E)$ of the graph $G$ is obtained from $G$ by merging all multiple edges, deleting all loops, and then deleting all isolated vertices.
It is non-empty by hypothesis and simple by construction.
\end{dfn}


\begin{dfn}[$2$-connectivity]\label{34}
By a \emph{nexus} of a graph $G=(\V,\E)$ we mean a vertex $v\in\V$ whose deletion from $G$ results in a disconnected graph $G-v$.\footnote{Equivalently, $v$ is a nexus if $\set{v}$ is a vertex-cut. If $G$ is connected, this means that $v$ is a cut-vertex.}
A connected graph without nexi is called \emph{$2$-(vertex-)connected}.
\end{dfn}


\begin{dfn}[Fat nexi]\label{0}
Let $v_0\in\V$ be a vertex with \emph{neighborhood}
\[
\V_0:=\set{v_0}\cup\set{v\in\V\mid\set{v,v_0}\in\E}\subseteq\V.
\]
We call $G$ a \emph{cone} with \emph{apex} $v_0$ if $\V_0=\V$.
For any subset $\U\subseteq\V$, set $\U^0:=\U\setminus\V_0$.
Then $v_0$ is called a \emph{fat nexus} in $G$ if it permits a partition
\[
\V=\set{v_0}\sqcup\V_1\sqcup\V_2
\]
such that the following conditions are satisfied:
\begin{enumerate}[(a)]
\item\label{0a} For $i\in\set{1,2}$, we have $\V_i\ne\emptyset$.
\item\label{0b} All edges between $\V_1$ and $\V_2$ have both vertices in $\V_0$.
In other words, there are no edges between $\V_i^0$ and $\V_j$ for $\set{i,j}=\set{1,2}$.
\item\label{0c} If $G$ is a cone with apex $v_0$, then $\abs{\V_1}\ne\abs{\V_2}$.
\end{enumerate}
\end{dfn}

The somewhat artificial condition~\eqref{0c} of Definition~\ref{0} will be used to address a special case in the proof of Theorem~\ref{2}.


\begin{rmk}\label{35}

We add some interpretation to the notions above.

\begin{asparaenum}[(a)]

\item\label{35a} The presence of a fat nexus implies $\abs{\V}\ge3$.

\item\label{35e} If $G$ is a cone with apex $v_0\in\V$ and $\abs{\V}\ge4$, then $v_0$ is a fat nexus:
Pick $v_1\in\V\setminus\set{v_0}$ and set $\V_1:=\set{v_1}$ and $\V_2:=\V\setminus\set{v_0,v_1}$. 

\item\label{35b} 
If $G$ is connected with $\abs{\V}\ge4$, then any nexus is fat.
Conversely, if $v_0\in\V$ is a fat nexus and no edges connect $\V_1$ and $\V_2$, then $v_0$ is a nexus.

\item\label{35c} The vertices in $\V\setminus\wt\V$ form the connected component singletons of $G$.
If $\wt G$ is disconnected or $G$ has at least $3$ connected components, then any vertex of $G$ is a fat nexus.

\item\label{35d} Simplification does not affect the existence of a fat nexus:
If $v_0\in\V\setminus\wt\V$ is a fat nexus of $G$, then $\wt G$ is disconnected and both $G$ and $\wt G$ have fat nexi by \eqref{35c}.
For any $v_0\in\wt\V$, being a fat nexus is equivalent for $G$ and $\wt G$.

\end{asparaenum}
\end{rmk}


\begin{exa}[$\star$-graphs]\label{36}
Suppose that $G=(\V,\E)$ is a planar connected graph with $\abs\V\ge4$ whose dual graph $G^\perp$ is a $\star$-graph in the sense of M\"uller-Stach and Westrich (see Definition~\ref{29}).
Then $G$ has a fat nexus (see Proposition~\ref{28}) and $\psi_G$ is the graph polynomial considered in loc.~cit.~(see Remark~\ref{12}.\eqref{12e}).
\end{exa}


Fix a field $\KK$.
Denote by $\K_0(\Var_\KK)$ the \emph{Grothendieck ring} of varieties
over $\KK$ (see \cite[\S12]{BB03}).
We write $[-]$ for classes in $\K_0(\Var_\KK)$, and denote by
\[
\LL:=[\AA^1]\in\K_0(\Var_\KK),\qquad\TT:=[\GG_m]\in\K_0(\Var_\KK)
\]
the \emph{Lefschetz motive} and the class of the \emph{$1$-torus} $\GG_m=\Spec\KK[t^{\pm1}]$ respectively. 
Then our main result is the following


\begin{thm}\label{10}
Let $G$ be a graph such that $\wt G$ has a nexus or a fat nexus.
Then the class of the graph hypersurface $X_G\subseteq\PP\KK^\E$ in the Grothendieck ring $\K_0(\Var_\KK)$ satisfies
\[
[X_G]\equiv\abs{\E}\mod\TT.
\]
Equivalently the class of its complement $Y_G=\PP\KK^\E\setminus X_G$ satisfies
\[
[Y_G]\equiv0\mod\TT.
\]
\end{thm}


\begin{rmk}\label{45}
If $\abs{\wt\V}\ge3$, then the hypothesis of Theorem~\ref{10} is that $\wt G$ has a fat nexus if it is $2$-connected.
\end{rmk}


\begin{cor}\label{41}
For a graph $G$ as in Theorem~\ref{10} and $\KK=\CC$, the Euler characteristic of $X_G$ equals $\chi(X_G)=\abs{\E}$, and hence $\chi(Y_G)=0$.\qedhere
\end{cor}


\begin{rmk}[Reduction to connected simple graphs]\label{30}
By Definitions~\ref{33} and \ref{9}, deleting isolated vertices does not affect $\psi_G\in\Sym(\KK^\E)^\vee$ and  $X_G\subseteq\PP\KK^\E$.
By Proposition~\ref{50}.\eqref{50a} deleting loops and merging multiple edges does not affect $[Y_G]\mod\TT$.
It follows that 
\[
[Y_{\vphantom{\wt G}\smash{G}}]\equiv[Y_{\wt G}]\mod\TT.
\]
This reduces the proof of Theorem~\ref{10} to the case of simple graphs.
If $G$ is disconnected, then Lemmas~\ref{11}.\eqref{11b} and \ref{8}.\eqref{8b} yield the claim (see Remark~\ref{66}).
\end{rmk}


\begin{cor}\label{42}
Aluffi's Conjecture~\ref{40} holds for all graphs whose simplification has a nexus or a fat nexus.
\end{cor}

\begin{proof}
By Remark \ref{30} and since $\chi(\TT)=0$, we may assume that $G=\wt G$.
Then $G$ is simple without isolated vertices.
If $\abs\V=2$, then $G=K_2$, $X_G=\emptyset$, $Y_G\cong\PP\KK$ is a point and $\chi(Y_G)=1$ (see Example~\ref{53}).
Otherwise, Corollary~\ref{41} applies to complete the proof.
\end{proof}


The following result disproves Aluffi's Conjecture in a strong sense.
Its proof relies on wheel graphs with all edges except one spoke subdivided into two edges (see Figure~\ref{62}).


\begin{thm}\label{70}
For each $n\in\ZZ$ there is a graph $G$ such that $[Y_G]\equiv n\mod\TT$ in the Grothendieck ring $\K_0(\Var_\KK)$.
\end{thm}

\begin{proof}
See Examples~\ref{58} and \ref{60}.
\end{proof}


Finally there is a counter-example $G$ to Aluffi's Conjecture~\ref{40} without edges in series.
Its particular feature is that it does (necessarily) not have a fat nexus, while its dual $G^\perp$ is a cone (see Figure~\ref{76}).
The calculation uses Theorem~\ref{10} and the formulas we derive in \S\ref{91} and \S\ref{92}.
\begin{arxiv}
It was performed by the implementation in Appendix~\ref{103} and verified by hand.
\end{arxiv}
\begin{journal}
It was performed in Python\footnote{The Python code can be downloaded as part of \cite{DPSW20}.}, using the package NetworkX, and independently verified by hand.
\end{journal}

\begin{exa}[A counter-example to Aluffi's Conjecture]\label{86}
For the graph $G$ in Figure~\ref{76}, we have $[Y_G]\equiv-2\mod\TT$ in the Grothendieck ring $\K_0(\Var_\KK)$.
\end{exa}

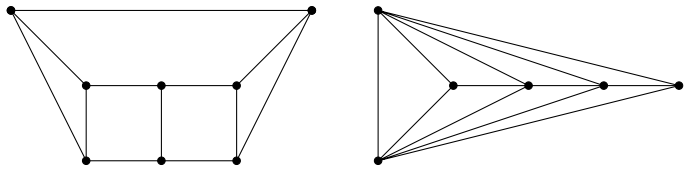
\begin{figure}[ht]
\begin{tikzpicture}[scale=1,baseline=(current bounding box.center)]
\tikzstyle{every node}=[circle,draw,inner sep=1pt,fill=black]
\draw (0,2) node {} -- (1,1) node {} -- (2,1) node {} -- (3,1) node {} -- (4,2) node {} -- cycle;
\draw (0,2) node {} -- (1,0) node {} -- (2,0) node {} -- (3,0) node {} -- (4,2) node {};
\draw (1,0) -- (1,1);
\draw (2,0) -- (2,1);
\draw (3,0) -- (3,1);
\end{tikzpicture}\qquad
\begin{tikzpicture}[scale=1,baseline=(current bounding box.center)]
\tikzstyle{every node}=[circle,draw,inner sep=1pt,fill=black]
\draw (1,1) -- (4,1);
\draw (0,2) node {} -- (0,0) node {};
\draw \foreach \x in {1,...,4} {(0,0) -- (\x,1) node {} -- (0,2)};
\end{tikzpicture}
\caption{The graph $G$ and its dual $G^\perp$.}\label{76}
\end{figure}

\section{Fat nexi and $\star$-graphs}\label{89}

Recall that the \emph{cycle space} 
\[
\C(G):=H_1(G,\FF_2)\subseteq\FF_2^\E
\]
of a graph $G=(\V,\E)$ is generated by the cycles in $G$ (see \cite[\S1.9]{Die18}).
The bijection of the vector space $\FF_2^\E$ with the power set $2^\E$, interpreting an element of the former as an indicator vector for an element of the latter, turns addition into symmetric difference; we use this translation freely in the following. 
A subset of $\FF_2^\E$ is called \emph{sparse} if each $e\in\E$ belongs to at most two of its elements.


We adopt the following notion of $\star$-graph from M\"uller-Stach and Westrich (see \cite[Def.~6]{MSW15}).


\begin{dfn}\label{29}
A connected graph $G=(\V,\E)$ is \emph{polygonal} if its edge set
\[
\E=\Delta_1\cup\dots\cup\Delta_h
\]
is the union of a sparse set of cycles $\set{\Delta_1,\ldots,\Delta_h}\subseteq\C(G)$ such that the edge sets
\[
(\Delta_1\cup\dots\cup\Delta_i)\cap\Delta_{i+1}\ne\emptyset
\]
induce (non-empty) connected graphs for all $i=1,\dots,h-1$.
If for every such polygonal decomposition the graph $F$ induced by the gluing set 
\[
\F:=\bigcup_{i\ne j}(\Delta_i\cap\Delta_j)
\]
is a forest, then $G$ is called a \emph{$\star$-graph}.
\end{dfn}


\begin{rmk}\label{25}
Note that $\star$-graphs are $2$-connected (see \cite[Prop.~3.1.1]{Die18}) and that $2$-connectivity is invariant under duality (see \cite[\S4,~Ex.~39]{Die18}).
\end{rmk}


\begin{prp}[$\star$-graphs and fat nexi]\label{28}
Let $G=(\V,\E)$ be a planar connected graph with $\abs\V\ge4$ whose dual graph $G^\perp=(\V^\perp,\E^\vee)$ is a $\star$-graph.
Then $G$ is a cone.
In particular, the apex is a fat nexus in $G$.
\end{prp}


\begin{proof}
As $G^\perp$ is a $\star$-graph, there are cycles $\Delta_1,\dots,\Delta_h$ fitting Definition~\ref{29}.
They form a sparse basis of $\C(G^\perp)$ (see \cite[Lem.~8.(i)]{MSW15}).
Due to sparsity, the corresponding gluing set $\F^\vee$ consists of all elements of $\E$ that belong to exactly two of the $\Delta_i$; the remaining elements of $\E$ belong to exactly one $\Delta_i$.
The complement $\Delta_0:=\E^\vee\setminus\F^\vee$ is therefore the symmetric difference of $\Delta_1,\ldots,\Delta_h$ and hence a disjoint union of cycles (see \cite[Prop.~1.9.1]{Die18}).
Then $\Delta_0,\Delta_1,\ldots,\Delta_h$ generate $\C(G^\perp)$ and each $e\in\E$ belongs to exactly two of them.
Thus $G^\perp$ is planar by MacLane's theorem and embeds into the $2$-sphere turning the cycles $\Delta_0,\Delta_1,\ldots,\Delta_h$ into face boundaries (see \cite[p.~109]{Die18}).
Since the number of faces in any planar embedding equals $\dim\C(G^\perp)+1=h+1$, $\Delta_0$ must be a single cycle.
The embedding gives rise to a bijection $\theta\colon\set{\Delta_0,\Delta_1,\ldots,\Delta_h}\to\V$ between face boundaries of $G^\perp$ and vertices of $G$.
Set $v_0:=\theta(\Delta_0)\in\V$. 
By definition the gluing graph $F^\vee$ induced by $\F^\vee$ is a forest.
In particular, it does not contain any $\Delta_i$ and hence $\Delta_i\cap \Delta_0\ne\emptyset$ for all $i\in\set{1,\ldots,h}$. 
This yields edges $\set{v_0,\theta(v_i)}\in\E$ for all $i\in\set{1,\ldots,h}$. 
Thus $\V=\V_0$ is the neighborhood of $v_0$, $G$ is a cone and the apex $v_0$ a fat nexus (see Remark~\ref{35}.\eqref{35e}).
\end{proof}

\section{Configurations and hypersurfaces}\label{93}

We extend our setup to prepare for the following sections:
Let $\E$ be any finite non-empty set, a special case being that of the edges of a graph $G=(\V,\E)$.

\begin{dfn}[Configurations]\label{33}
A \emph{configuration} is a subspace
\[
W\subseteq\KK^\E=:V
\]
where $\E$ is identified with a basis of $V$.
Pick an orientation on the edges $\E$ of $G$ to consider each $v\in\V$ also as an incidence vector 
\[
v=(v_e)_{e\in\E}\in\set{-1,0,1}^\E\subseteq\KK^\E=V,
\]
where $v_e=-1$ or $v_e=1$ signifies that $v$ is respectively the source or target of the non-loop edge $e\in\E$, and $v_e=0$ in all other cases.
The $\KK$-span of these vectors is the \emph{graph configuration}
\[
W=W_G:=\ideal{\V}\subseteq V.
\]
\end{dfn}


\begin{rmk}\label{13}
Since $\sum_{v\in\V}v=0$, $W=\ideal{\V\setminus\set{v}}$ for any $v\in\V$.
Indeed, $\V\setminus\set{v}$ is a basis of $W$ if $G$ is connected.
In general $W\subseteq V$ is a direct sum of corresponding spaces constructed from the connected components of $G$.
\end{rmk}


\begin{ntn}[Dual space]\label{14}
Let $\E^\vee=\set{e^\vee\mid e\in\E}$ denote the \emph{dual basis} of $\E$ defined by $e^\vee(f):=\delta_{e,f}$, where $\delta$ is the Kronecker symbol.
This identifies the \emph{dual space} $V^\vee =(\KK^\E)^\vee$ with $\KK^{(\E^\vee)}$.
Writing $x_e:=e^\vee$, we consider $\E^\vee$ as a coordinate system $x=x_\E=(x_e)_{e\in\E}$ on $V$.
Then $w_e:=x_e(w)$ is the $e$-coordinate of $w\in V$.
The distinguished bases of $V$ and $V^\vee$ gives rise to a isomorphism 
\[
q\colon V\to V^\vee,\quad w=\sum_{e\in\E}w_e\cdot e\mapsto\sum_{e\in\E}w_e\cdot x_e.
\]
For $S\subseteq\E$, set 
\[
S^\perp:=q(\E\setminus S)\subseteq\E^\vee
\]
and denote the monomial obtained from $x_S:=(x_e)_{e\in S}$ by 
\[
x^S:=\prod_{e\in S}x^e\in\Sym V^\vee=\KK[V].
\]
\end{ntn}


\begin{dfn}[Matroids]\label{67}
The linear dependence relations on $\E$ obtained by mapping
\[
\E\ni e\mapsto e^\vee\vert_W\in W^\vee
\]
define the \emph{matroid} $\M=\M_W$ of $W\subseteq V$, or $\M_G:=\M_{W_G}$ of $G$.
In this case $W$ is a \emph{realization} of $\M$ and $\M$ is called \emph{realizable}.
All matroids we consider are realizable.
We denote respectively by $\I_\M$, $\B_\M$, $\C_\M$ and $\L_\M$ the set of \emph{independent sets}, \emph{bases} and \emph{circuits}, and the lattice of \emph{flats} of $\M$.
We set 
\[
b(\M):=\abs{\B_\M}.
\]
We write $\cl=\cl_\M$, $\rk=\rk_\M$ and $\nullity=\nullity_\M$ for the \emph{closure}, \emph{rank} and \emph{nullity} operators of $\M$.
Recall that $\nullity(S)=\abs{S}-\rk S$ for $S\subseteq\E$.
The \emph{dual matroid} $\M^\perp$ of $\M$ on $\E^\vee$ has bases $\B_{\M^\perp}=\set{B^\perp\mid B\in\B_\M}$.
For further matroid theory and notation we refer to Oxley (see \cite{Oxl11}).
\end{dfn}


\begin{rmk}[Parallels and series]\label{47}
We recall that $e,f\in\E$ are \emph{parallel} in $\M$ if $\set{e,f}\in\C_\M$, and \emph{in series} if $e^\vee,f^\vee\in\E^\vee$ are parallel in $\M^\perp$. 
Note that, if $e,f$ are parallel (in series), then either $e=f$ is a (co)loop, or $e\ne f$ are both non-(co)loops.
If $\M=\M_G$, then $e,f$ in series means that $\set{e,f}$ is a minimal edge-cut (see \cite[Prop.~2.3.1]{Oxl11}).
Note that a \emph{subdivided edge} is just a particular case of two edges in series (see \cite[Fig.~5.13]{Oxl11}).
\end{rmk}


\begin{rmk}[Connectivity]\label{66}
If $G=(\V,\E)$ is a loopless graph with $\abs{\V}\ge3$ and without isolated vertices, then $G$ is $2$-connected if and only if $\M_G$ is ($2$-)connected (see \cite[Prop.~4.1.7]{Oxl11}).
\end{rmk}


\begin{rmk}[Operations]\label{48}
There are vector space operations on configurations that induce the matroid operations of \emph{restriction} or \emph{deletion}, \emph{contraction} and \emph{duality} (see \cite[Def.~2.17]{DSW21}). 
They are compatible with those on graphs in case of the graph configuration.
\end{rmk}


\begin{dfn}[Hypersurfaces]\label{9}
Consider the symmetric bilinear form 
\[
\Sym^2(V^\vee)\otimes V^\vee\ni Q:=\sum_{e\in\E}e^\vee\cdot e^\vee\cdot x_e\colon V\times V\to V^\vee.
\]
Its restriction to $W\times W$ is the \emph{configuration form} of $W$,
\[
\Sym^2(W^\vee)\otimes V^\vee\ni Q_W\colon W\times W\to V^\vee,
\]
or the \emph{graph form} $Q_G:=Q_{W_G}$ of $G$.
Its determinant with respect to some choice of basis of $W$ is the \emph{configuration polynomial} of $W$ (defined up to a factor in $\KK^*$),
\[
\psi_W:=\det Q_W\in\Sym V^\vee=\KK[V],
\]
or the \emph{Kirchhoff polynomial} $\psi_G:=\psi_{W_G}$ of $G$ (see \cite[Prop.~3.16]{DSW21}).
It defines the \emph{(projective) configuration/graph hypersurface} and its complement
\begin{alignat*}{2}
X_W&:=V(\psi_W)\subseteq\PP V,&\quad X_G&:=X_{W_G}=V(\psi_G)\subseteq\PP V,\\
Y_W&:=\PP V\setminus X_W,&\quad Y_G&:=Y_{W_G}=\PP V\setminus X_G.
\end{alignat*}
\end{dfn}


\begin{rmk}[Equivalence]\label{54}
If $W$ and $W'$ are equivalent configurations, then $\psi_W$ and $\psi_{W'}$ differ only by scaling variables and hence $X_W\cong X_{W'}$ and $Y_W\cong Y_{W'}$ (see \cite[Rem.~3.4]{DSW21}).
If $\M_W=\M_G$, then $W$ is equivalent to $W_G$ and hence $X_W\cong X_G$ and $Y_W\cong Y_G$ (see \cite[Rem.~3.6]{DSW21}).
\end{rmk}


\begin{ntn}[Hadamard product]\label{46}
The \emph{Hadamard product} of $w,w'\in V$ with respect to $\E$ is denoted by
\[
w\star w':=\sum_{e\in\E}w_e\cdot w'_e\cdot e.
\]
\end{ntn}


\begin{rmk}\label{12}\
\begin{enumerate}[(a)]

\item\label{12c} For $w,w'\in V$,
\[
Q(w,w')=\sum_{e\in\E}w_e\cdot w'_e\cdot x_e=q(w\star w').
\]

\item\label{12a} If $\rk\M_W=0$ which means that $\E$ contains loops only, then $\psi_W=1$ (the determinant of the $0\times0$-matrix) and hence $X_W=\emptyset$, $Y_W=\PP V$ and $\chi(Y_W)=\abs{\E}$.
This justifies excluding graphs made of loops only.

\item\label{12d} For some choice of basis of $W$ (see \cite[Lem.~1.3]{BEK06}),
\[
\psi_W=\sum_{B\in\B_{\M_W}}\det(W\onto\KK^B)^2\cdot x^B.
\]
In case of a graph $G$, $\B_{\M_G}=\T_G$ is the set of spanning forests in $G$ and
\[
\psi_G=\sum_{T\in\T_G}x^T
\]
is the matroid (basis) polynomial of $\M_G$.

\item\label{12e} Let $G$ be a connected planar graph with dual graph $G^\perp$.
Then 
\[
\psi_G=\sum_{T\not\in\T_{G^\perp}}x^T
\]
is the graph polynomial associated to $G^\perp$ by M\"uller-Stach and Westrich (see \cite[Def.~1]{MSW15}).
Note that the associated matroid $\M_{G^\perp}=(\M_G)^\perp$ does not depend on the planar embedding of $G$ used to construct $G^\perp$.

\item\label{12b} $X_W$ is a reduced algebraic scheme over $\KK$ (see \cite[Thm.~4.16]{DSW21}).

\end{enumerate}
\end{rmk}


\begin{exa}\label{53}
If $G$ is a graph with $\abs\V=2$ vertices, then $X_G\subseteq\PP\KK^\E$ is a
hyperplane and $Y_G\cong\AA^{\abs\E-1}$ is an affine space with $[Y_G]=1$.
\end{exa}

\section{Loops, parallels, and disconnections}\label{91}

In this section we deal with the most elementary reductions for the class of $Y_G$, namely for graphs $G$ that have a loop, coloop, multiple edge or disconnection. 
We start in Lemma~\ref{11} with a discussion of the shape of the underlying configuration polynomial.
Lemma~\ref{8} and Proposition~\ref{50} translate the algebraic information into geometry.


\begin{lem}\label{11}\ 

\begin{enumerate}[(a)]

\item\label{11a} If $e\in\E$ is a loop or coloop in $\M_W$, then respectively $\psi_W=\psi_{W\setminus e}\cdot\psi_0$ or $\psi_W=\psi_{W\setminus e}\cdot\psi_{\KK^{\set{e}}}$ where $\psi_0=1$ and $\psi_{\KK^{\set{e}}}=x_e$.

\item\label{11c} If nonloops $e,f\in\E$ are parallel in $\M_W$, then $\psi_W$ is obtained from $\psi_{W\setminus e}$ by substituting $x_f$ by $x_e+x_f$, up to scaling variables.

\item\label{11b} If $\M_W$ is disconnected, then there is a proper partition $\E=\E_1\sqcup\E_2$ such that $W=W_1\oplus W_2$ where $W_i:=W\cap\KK^{\E_i}$, $i=1,2$, and
\[
\psi_W=\psi_{W_1}\cdot\psi_{W_2}.
\]
In particular, if $G$ has no isolated vertices, and is disconnected or has a nexus, then 
\[
\psi_G=\psi_{G_1}\cdot\psi_{G_2}.
\]
for the edge-induced subgraphs $G_1=(\V_1,\E_1)$ and $G_2=(\V_2,\E_2)$.

\end{enumerate}
\end{lem}

\begin{proof}\
\begin{asparaenum}[(a)]

\item By hypothesis (see Remark~\ref{48})
\[
W=(W\setminus e)\oplus W'\subseteq\KK^{\E\setminus\set{e}}\oplus\KK^{\set{e}}=\KK^\E,
\]
where $W'=0$ or $W'=\KK^{\set{e}}$ respectively, and the claim follows (see Definition~\ref{9} and Remark~\ref{12}.\eqref{12a}).

\item By hypothesis $e^\vee\vert_W$ and $f^\vee\vert_W$ are collinear.
So the statement already holds for $Q_W$, and hence for $\psi_W$ (see Definition~\ref{9}).

\item For the first statement we refer to \cite[Prop.~3.12]{DSW21}.

If $G$ is disconnected, the claim is straightforward.
A nexus $v_0\in\V$ gives rise to a desired partition such that $\V_1\cap\V_2=\set{v_0}$.
Then (see Remark~\ref{13})
\begin{equation}\label{22}
W_i=W_{G_i}=\ideal{\V_i\setminus\set{v_0}}\subseteq\KK^{\E_i},\quad i\in\set{1,2},
\end{equation}
and hence
\begin{align}\label{21}
W=\ideal{\V}=\ideal{\V\setminus\set{v_0}}&=\ideal{\V_1\setminus\set{v_0}}\oplus\ideal{\V_2\setminus\set{v_0}}\\
&=W_1\oplus W_2\subseteq\KK^{\E_1}\oplus\KK^{\E_2}=\KK^\E.\nonumber
\end{align}
It follows that (see Remark~\ref{12}.\eqref{12c})
\[
Q_G=
\begin{pmatrix}
Q_{G_1} & 0\\
0 & Q_{G_2}
\end{pmatrix}
\]
and hence the claim.\qedhere

\end{asparaenum}
\end{proof}


\begin{lem}\label{8}
For $m,n\in\NN$, set $x=x_0,\dots,x_m$ and $y=y_0,\dots,y_n$.
Let $f\in\KK[x]\setminus\KK$ and $g\in\KK[y]\setminus\set{0}$ be homogeneous polynomials.
Consider the projective hypersurfaces
\[
X=V(f)\subseteq\PP^m,\qquad Y=V(g)\subseteq\PP^n,\qquad Z=V(f\cdot g)\subseteq\PP^{m+n+1}.
\] 
\begin{enumerate}[(a)]

\item\label{8a} If $g\in\KK$, then 
\[
[Z]=[X]\cdot[\AA^{n+1}]+[\PP^n]\in\K_0(\Var_\KK).
\]
In particular, $[Z]\equiv[X]+n+1\mod\TT$ and $[\PP^{m+n+1}\setminus Z]\equiv[\PP^m\setminus X]\mod\TT$.

\item\label{8b} If $g\not\in\KK$, then 
\[
[Z]=([X]\cdot[\PP^n]+[Y]\cdot[\PP^m]-[X]\cdot[Y])\cdot\TT+[\PP^m]+[\PP^n]\in\K_0(\Var_\KK).
\]
In particular, $[Z]\equiv m+1+n+1\mod\TT$ and $[\PP^{m+n+1}\setminus Z]\equiv0\mod\TT$.

\end{enumerate}
\end{lem}

\begin{proof}
For the particular claims note that
\begin{equation}\label{55}
[\PP^n]=\LL^0+\dots+\LL^n\equiv n+1\mod\TT.
\end{equation}
For the main claims write 
\begin{equation}\label{31}
\PP^{m+n+1}=E\cup\PP^n\cup F\cup\PP^m,\quad E:=\PP^{m+n+1}\setminus\PP^n,\quad F:=\PP^{m+n+1}\setminus\PP^m,
\end{equation}
where $E$ is an $\AA^{n+1}$-bundle over $\PP^m$, and $F$ an $\AA^{m+1}$-bundle over $\PP^n$.

\begin{asparaenum}[(a)]

\item By hypothesis, $f\not\in\KK$ and $g\in\KK^*$ and hence $\PP^n\subseteq V(f)=Z$.
It follows that
\[
\PP^m\cap Z=X\subseteq E\vert_X=E\cap Z,\quad
F\cap Z\subseteq E\vert_X\cup\PP^n.
\]
So \eqref{31} induces a decomposition $Z=E\vert_X\sqcup\PP^n$ and the claim follows.

\item By hypothesis, $f,g\not\in\KK$ and hence $\PP^m\cup \PP^n\subseteq V(f)\cup V(g)=Z$.
So \eqref{31} yields
\begin{equation}\label{26}
Z=E\vert_X\cup\PP^n\cup F\vert_Y\cup\PP^m.
\end{equation}
Since $\PP^m\cap\PP^n=\emptyset$, the only non-empty intersections are
\begin{equation}\label{27}
E\vert_X\cap\PP^m=X,\quad
F\vert_Y\cap\PP^n=Y,\quad
E\vert_X\cap F\vert_Y=V(f,g)\setminus(\PP^m\cup\PP^n).
\end{equation}
The latter is covered by affine open sets
\[
U_{i,j}:=V(f,g)\cap D(x_i\cdot y_j).
\]
Denote $V_{i,j}:=(X\cap D(x_i))\times(Y\cap D(y_j))$.
Then there are isomorphisms
\[
\begin{tikzcd}
U_{i,j}\ar[r] & V_{i,j}\times\GG_m\\[-20pt]
(x:y)\ar[r,mapsto] & (x,y,y_j/x_i),\\[-20pt]
(x/x_i\colon ty/y_j) & (x,y,t)\ar[l,mapsto].
\end{tikzcd}
\]
Over $V_{i,j}\cap V_{k,\ell}$, the transition maps are given by multiplication by
\[
\lambda^{i,j}_{k,\ell}=\frac{x_i}{y_j}\frac{y_\ell}{x_k}.
\]
Over $V_{i,j}\cap V_{k,\ell}\cap V_{r,s}$, these satisfy the cocycle condition
\[
\lambda^{i,j}_{k,\ell}\cdot\lambda^{k,\ell}_{r,s}=\lambda^{i,j}_{r,s}.
\]
It follows that $V(f,g)\setminus(\PP^m\cup\PP^n)$ is a locally trivial fibration over $X\times Y$ with fiber $\GG_m$.
Using \eqref{26} and \eqref{27}, this yields in $\K_0(\Var_\KK)$ the identity
\begin{align*}
[Z]&=[X]\cdot\LL^{n+1}+[Y]\cdot\LL^{m+1}+[\PP^m]+[\PP^n]-[X]-[Y]-[X]\cdot[Y]\cdot\TT\\
&=[X]\cdot(\LL^{n+1}-1)+[Y]\cdot(\LL^{m+1}-1)-[X]\cdot[Y]\cdot\TT+[\PP^m]+[\PP^n]
.
\end{align*}
The claim then follows since
\begin{equation}\label{51}
[\PP^n]\cdot\TT=(\LL^0+\dots+\LL^n)\cdot(\LL-1)=\LL^{n+1}-1.\qedhere
\end{equation}

\end{asparaenum}
\end{proof}


\begin{prp}\label{50}\ 
\begin{enumerate}[(a)]

\item\label{50a} If $\rk\M_W>0$ and $e\in\E$ is a loop or parallel to some $f\in\E$ in $\M_W$, then $[Y_W]=[Y_{W\setminus e}]\cdot\LL\in\K_0(\Var_\KK)$.
In particular, $[Y_W]\equiv[Y_{W\setminus e}]\mod\TT$.

\item\label{50c} If $\rk\M_W>1$ and $e\in\E$ is a coloop in $\M_W$, then $[Y_W]=[Y_{W\setminus e}]\cdot\TT\in\K_0(\Var_\KK)$.
In particular, $[Y_W]\equiv0\mod\TT$.

\item\label{50b} If $\M_W$ is loopless and disconnected, then $[Y_W]\equiv0\mod\TT$ in $\K_0(\Var_\KK)$.
In particular, if $G$ is loopless without isolated vertices, and is disconnected or has a nexus, then $[Y_G]\equiv0\mod\TT$ in $\K_0(\Var_\KK)$.

\end{enumerate}
\end{prp}

\begin{proof}\
\begin{asparaenum}[(a)]

\item This is immediate from Lemmas~\ref{11}.\eqref{11a}, \eqref{11c} and \ref{8}.\eqref{8a}.

\item Apply Lemmas~\ref{11}.\eqref{11c} and \ref{8}.\eqref{8b} with $m=0$, $f=x_e$ and $g=\psi_{W\setminus\set{e}}$ and hence $[X]=0$ and $[\PP^m]=1$.
Then 
\[
[X_W]=[X_{W\setminus\set{e}}]\cdot\TT+[\PP^0]+[\PP^n]
\]
and hence using \eqref{55} and \eqref{51}
\begin{align*}
[Y_W]&=[\PP^{n+1}]-[X_W]=[\PP^{n+1}]-[\PP^n]-[\PP^0]-[X_{W\setminus\set{e}}]\cdot\TT\\
&=\LL^{n+1}-1-[X_{W\setminus\set{e}}]\cdot\TT=([\PP^n]-[X_{W\setminus\set{e}}])\cdot\TT=[Y_{W\setminus\set{e}}]\cdot\TT.
\end{align*}

\item This is immediate from Lemmas~\ref{11}.\eqref{11b} and \ref{8}.\eqref{8b}.\qedhere

\end{asparaenum}
\end{proof}

In Proposition~\ref{50}.\eqref{50c} one would expect also a statement for coparallel elements, that is elements in series. 
We postpone this to Corollary~\ref{59}.

\section{Torus actions from fat nexi}\label{94}

In this section we discuss how a fat nexus in $G$ enables a non-monomial torus action on $X_G$. 
This relies on a decomposition of the first and second Hadamard powers of $W$ in Lemma~\ref{1} and a discussion of fixed points in Theorem~\ref{2}.
We combine both with the Theorem of Bia{\l}ynicki-Birula in order to provide a proof for Theorem \ref{10}.


\begin{lem}\label{1}
Let $G=(\V,\E)$ be a connected simple graph with a vertex partition $\V=\set{v_0}\sqcup\V_1\sqcup\V_2$ making $v_0\in\V$ a fat nexus.
Setting $W_i:=\ideal{\V_i}\subseteq\KK^\E$, $i=1,2$, gives rise to non-trivial direct sum decompositions
\[
W=W_1\oplus W_2,\quad W\star W=W_1\star W_1\oplus W_1\star W_2\oplus W_2\star W_2.
\]
\end{lem}

\begin{proof}
Define $\ol\E\subseteq\E$ by deleting all edges between $\V_1$ and $\V_2$ and leave $\V$ unchanged. 
Consider the configuration $\ol W:=W_{\ol G}$ of the graph $\ol G=(\V,\ol\E)$ and set $\ol W_i:=\ideal{\V_i}\subseteq\KK^{\ol\E}$, $i=1,2$.
By Definition~\ref{0}.\eqref{0b} each deleted edge has vertices in $\V_0$ and hence closes a triangle.
The projection 
\[
\pi\colon\KK^\E\onto\KK^{\ol\E}
\]
thus induces an isomorphism $W\cong\ol W$ and hence isomorphisms $W_i\cong \ol W_i$, $i=1,2$.
In $\ol G$, $v_0$ is a nexus and hence $\ol\E=\ol\E_1\sqcup\ol\E_2$ such that (see \eqref{22} and \eqref{21})
\begin{equation}\label{37}
\ol W=\ol W_1\oplus\ol W_2\subseteq\KK^{\ol\E_1}\oplus\KK^{\ol\E_2}=\KK^{\ol\E}.
\end{equation}
The direct sum decomposition of $W$ is then induced via $\pi$.
Since the Hadamard product is bilinear and symmetric,
\begin{equation}\label{38}
W\star W=W_1\star W_1+W_1\star W_2+W_2\star W_2.
\end{equation}
By \eqref{37}, the Hadamard product $\olstar$ in $\KK^{\ol\E}$ satisfies 
\begin{equation}\label{44}
\ol W\olstar\ol W=\ol W_1\olstar\ol W_1\oplus\ol W_2\olstar\ol W_2.
\end{equation}
This proves the claim in case $G=\ol G$.
We reduce the general case to this latter case using the fat nexus $v_0$.
To this end, consider a zero linear combination of generators (see Remark~\ref{13}) of the summands in \eqref{38}:
\begin{equation}\label{43}
0=\ell:=\sum_{v,v'\in\V\setminus\set{v_0}}\lambda_\set{v,v'}\cdot v\star v'.
\end{equation}
By Definition~\ref{33}, $v\star v'\ne0$ only if $\set{v,v'}=\set v=v\in\V$ or if $\set{v,v'}\in\E$.
For $e=\set{v,v'}\in\E$ where $v\in\V\setminus\set{v_0}$ and $v'\in\V$, consider the projection
\[
\pi_e\colon\KK^\E\onto\KK^e\cong\KK.
\]
Applying $\pi_e$ to \eqref{43} when $v'=v_0$, and therefore $v\in\V_0\setminus\set{v_0}$, shows that
\[
0=\pi_e(\ell)=\lambda_v\cdot(v\star v)_e=\lambda_v.
\]
With this in hand, applying $\pi_e$ to \eqref{43} when $e\subseteq\V_0\setminus\set{v_0}$ and noting that $\lambda_v=0=\lambda_{v'}$, then yields
\[
0=\pi_e(\ell)=\lambda_v\cdot(v\star v)_e+\lambda_{v'}\cdot(v'\star v')_e-\lambda_e\cdot(v\star v')_e=-\lambda_e.
\]
So $\lambda_\set{v,v'}=0$ for all $v,v'\in\V_0\setminus\set{v_0}$.
By Definition~\ref{0}.\eqref{0b} this applies to all $v\in\V_1$ and $v'\in\V_2$ and eliminates all terms of $\ell$ in $W_1\star W_2$.
Any remaining $v\star v'$ with $\lambda_\set{v,v'}\ne0$ has indices $v,v'\in\V_i$ for some $i\in\set{1,2}$ with $\set{v,v'}\not\subseteq\V_0$.
It follows that
\[
W_i\star W_i\ni v\star v'=\pi(v)\olstar\pi(v')\in\ol W_i\olstar\ol W_i.
\]
Due to the direct sum in \eqref{44} then also the sum in \eqref{38} is direct.
\end{proof}


In the sequel, all schemes are over $\KK$.
For lack of suitable reference we describe the fixed point schemes of torus actions on projective space. 


\begin{lem}\label{5}
Suppose that $\GG_m$ acts linearly through distinct characters $\chi_1,\dots,\chi_s$ on the direct summands of the finite dimensional vector space
\[
V=V_1\oplus\dots\oplus V_s.
\]
Then there is an induced $\GG_m$-action on $\PP V$ with fixed point scheme 
\[
(\PP V)^{\GG_m}=\PP V_1\sqcup\cdots\sqcup\PP V_s.
\]
\end{lem}

\begin{proof}
For any $\KK$-algebra $A$, the $A$-valued points $L\in\PP V(A)$ are direct summands of $V\otimes A=V(A)$ of rank $1$ (see \cite[\S 7.d]{Mil17} and \cite[Part~I, §2.2]{Jan03}).
Considering $\chi_i\in\ZZ$ (see \cite[Part~I, §2.5]{Jan03}), $t\in\GG_m(A)=A^*$ acts on $V_i$ by multiplication by $t^{\chi_i}$ (see \cite[\S 4.g]{Mil17}).
For the induced $\GG_m$-action on $\PP V$ (see \cite[\S 7.b]{Mil17}) 
\[
(\PP V)^{\GG_m}(A)=\set{L\in\PP V(A)\mid\forall B\supseteq A\colon\forall t\in \GG_m(B)\colon t\bullet L=L}.
\]
This makes the inclusion \enquote{$\supseteq$} obvious.
Choosing $B$ to be an infinite field makes $L\subseteq V\otimes B$ a $1$-dimensional subspace and the inclusion \enquote{$\subseteq$} follows readily.
\end{proof}


For a connected simple graph with a fat nexus and $q$ from Notation~\ref{14}, Lemma~\ref{1} yields a direct sum decomposition
\[
q(W\star W)=q(W_1\star W_1)\oplus q(W_1\star W_2)\oplus q(W_2\star W_2)\subseteq V^\vee.
\]
After enlarging one of the direct summands by a complement, there is a (unique) dual decomposition
\begin{equation}\label{16}
V=V_1\oplus V_2\oplus V_3
\end{equation}
with respect to the canonical pairing $V^\vee\times V\to\KK$.


\begin{thm}\label{2}
Suppose that $G$ is a connected simple graph with a fat nexus and let $\GG_m$ act linearly through characters $0,1,2$ on the direct summands $V_1,V_2,V_3$ in \eqref{16}.
Then this action descends to $X_G$ with fixed point scheme
\begin{equation}\label{23}
X_G^{\GG_m}=\PP V_1\sqcup\PP V_2\sqcup\PP V_3.
\end{equation}
\end{thm}


\begin{proof}
With respect to the decomposition $W=W_1\oplus W_2$ from Lemma~\ref{1}, 
\begin{equation}\label{6}
Q_G=
\begin{pmatrix}
Q_{1,1} & Q_{1,2} \\
Q_{2,1} & Q_{2,2}
\end{pmatrix},\quad
Q_{i,j}:=Q_G\vert_{W_i\times W_j}\colon W_i\times W_j\to V^\vee.
\end{equation}
By Remark~\ref{12}.\eqref{12c} with $q$ from Notation~\ref{14},
\begin{equation}\label{7}
Q_{i,j}(W_i\times W_j)\subseteq q(W_i\star W_j).
\end{equation}
By construction of the decomposition \eqref{16},
\begin{equation}\label{15}
q(W_1\star W_1)\subseteq V_1^\vee,\quad
q(W_1\star W_2)\subseteq V_2^\vee,\quad
q(W_2\star W_2)\subseteq V_3^\vee
\end{equation}
and
\begin{equation}\label{4}
q(W_1\star W_1)\perp V_2\oplus V_3,\quad
q(W_1\star W_2)\perp V_1\oplus V_3,\quad
q(W_2\star W_2)\perp V_1\oplus V_2.
\end{equation}

There is an induced $\GG_m$-action on $\PP V$ (see Lemma~\ref{5}), and a natural right-$\GG_m$-module structure on the coordinate ring
\[
\KK[V]=\Sym V^\vee.
\]
By \eqref{7} and \eqref{15}, $t\in\GG_m$ acts on $Q_G$ and hence on $\psi_G=\det Q_G$ by
\[
Q_G\bullet t=
\begin{pmatrix}
Q_{1,1} & t\cdot Q_{1,2} \\
t\cdot Q_{2,1} & t^2\cdot Q_{2,2}
\end{pmatrix},\quad
\psi_G\bullet t=t^{2\dim W_2}\cdot\psi_G.
\]
This makes $\ideal{\psi_G}\unlhd\KK[V]$ a $\GG_m$-stable ideal which yields an induced $\GG_m$-action on $V(\psi_G)\subseteq V$ (see \cite[Part~I, §2.8]{Jan03}), 
and hence on $X_G=V(\psi_G)\subseteq\PP V$.

By Lemma~\ref{5}, $(\PP V)^{\GG_m}$ is the right hand side of \eqref{23}, and it suffices to show that $\PP V_i\subseteq X_G$ for $i=1,2,3$.
By \eqref{6}, \eqref{7} and \eqref{4}, restricting $Q_G$ in the target $V^\vee$ gives
\[
Q_G\vert_{V_1}=
\begin{pmatrix}
* & 0 \\
0 & 0
\end{pmatrix},\quad
Q_G\vert_{V_2}=
\begin{pmatrix}
0 & Q_{1,2}\vert_{V_2} \\
Q_{2,1}\vert_{V_2} & 0
\end{pmatrix},\quad
Q_G\vert_{V_3}=
\begin{pmatrix}
0 & 0 \\
0 & *
\end{pmatrix}.
\]
For $i=1,3$, $Q_G\vert_{V_i}$ is singular and hence $\psi_G\vert_{V_i}=0$ and $\PP V_i\subseteq X_G$.
Since $\dim(W_i)=\abs{\V_i}$, $i=1,2$, the same holds for $i=2$ if $\abs{\V_1}\ne\abs{\V_2}$.
If $\abs{\V_1}=\abs{\V_2}$, then $\V\ne\V_0$ by Definition~\ref{0}.\eqref{0c}.
We may assume that $\V_1^0\ne\emptyset$ and hence $\ideal{\V_1^0}\ne0$.
By Definition~\ref{0}.\eqref{0b}, $G$ has no edges between $\V_1^0$ and $\V_2$, hence any row of $Q_{1,2}\vert_{V_2}$ indexed by an element of $\V_1^0$ is zero.
So again $Q_G\vert_{V_2}$ in singular and $\PP V_2\subseteq X_G$, in both cases.
\end{proof}


\begin{proof}[Proof of Theorem~\ref{10}]
By Remark~\ref{30}, we may assume that $G=\wt G$ is simple and connected.
If $G$ has a nexus, then $\psi_G=\psi_{G_1}\cdot\psi_{G_2}$ decomposes as in Lemma~\ref{11}.\eqref{11b}.
Then both $G_1$ and $G_2$ contain a non-loop and hence $\psi_{G_1}$ and $\psi_{G_2}$ are non-constant.
Thus Lemma~\ref{8}.\eqref{8b} yields the claim in this case.
Suppose now that $G$ has no nexus, and hence a fat nexus.
By Remark~\ref{35}.\eqref{35a} then $\abs{V}\ge3$, $G$ is $2$-connected by definition, and hence the graphic matroid $\M_G$ is connected (see Remark~\ref{66}).
By \cite[Prop.~3.8]{DSW21}, $\psi_G$ is then irreducible and $X_G$ is an integral algebraic scheme over $\KK$.
Now the Theorem of Bia{\l}ynicki-Birula (see \cite[Thm.~2]{Bia73b} and \cite[Rem.~2.3]{Hu12}) applies to the $\GG_m$-action from Theorem~\ref{2}:
\begin{align*}
[X_G]&\equiv[X_G^{\GG_m}]\mod\TT\\
&=[(\PP V)^{\GG_m}]=[\PP V_1]+[\PP V_2]+[\PP V_3]\\
&=\dim V_1+\dim V_2+\dim V_3=\dim V=\abs{\E}.
\end{align*}
The class $[Y_G]$ of the complement is then zero modulo $\TT$ by  \eqref{55}.
\end{proof}

\section{Orbits, involution and duality}\label{92}

Our goal here is to compute the class of $Y_W$ modulo $\TT$ in $\K_0(\Var_\KK)$ using the toric stratification of $\PP V$ and duality of configurations $W$.


\begin{dfn}[Torus parts]
For each $S\subseteq\E$, consider the \emph{torus orbit}
\[
\GG_m^{\abs S-1}\cong O_S:=D(x^S)\cap V(x_{\E\setminus S})\subseteq\PP V.
\]
We will denote the respective \emph{torus parts} of $X_W$ and $Y_W$ by
\[
X_W^\circ:=X_W\cap O_\E,\quad 
Y^\circ_W:=Y_W\cap O_\E=O_\E\setminus X^\circ_W.
\]
\end{dfn}


The approach is based on the following facts (see \cite[Prop.~3.10, 3.12]{DSW21}).
We recall that the (standard) \emph{Cremona transformation} with chosen global coordinates $x_\E$ is the birational isomorphism
\[
\PP V\to\PP V^\vee
\]
defined by the assignment $x_{e^\vee}\mapsto x_e^{-1}$ for all $e\in\E$. 
It induces an isomorphism of open torus orbits $O_\E\cong O_{\E^\vee}$.


\begin{lem}[Involution and duality]\label{62}
The Cremona involution $O_\E\cong O_{\E^\vee}$ identifies $X_W^\circ\cong X_{W^\perp}^\circ$ and hence $Y_W^\circ\cong Y_{W^\perp}^\circ$.\qed
\end{lem}


\begin{lem}[Torus parts and restriction]\label{64}
For $\emptyset\ne S\subseteq\E$, we have
\[
\psi_W\vert_{x_{\E\setminus S}=0}=\psi_{W\vert_S}.
\]
In particular, we can identify
\begin{alignat*}{2}
X_W\cap V(x_{\E\setminus S})&\cong X_{W\vert_S},&\quad Y_W\cap V(x_{\E\setminus S})&\cong Y_{W\vert_S},\\
X_W\cap O_S&\cong X_{W\vert_S}^\circ,&\quad Y_W\cap O_S&\cong Y_{W\vert_S}^\circ.\tag*{\qed}
\end{alignat*}
\end{lem}


We further record a consequence of Lemma~\ref{11}.

\begin{lem}\label{52}\
\begin{enumerate}[(a)]

\item\label{52a} If $e\in\E\ne\set{e}$ is a loop or coloop in $\M_W$, then $[Y^\circ_W]=[Y^\circ_{W\setminus e}]\cdot\TT$.

\item\label{52c} If $e,f\in\E\ne\set{e,f}$ are either parallel nonloops or noncoloops in series in $\M_W$, then $[Y_W^\circ]+[Y_{W\setminus e}^\circ]\equiv0\mod\TT$ or $[Y_W^\circ]+[Y_{W/e}^\circ]\equiv0\mod\TT$ respectively.

\item\label{52b} If $\M_W$ is disconnected, then $[Y^\circ_W]\equiv0\mod\TT$.

\end{enumerate}
\end{lem}

\begin{proof}\
\begin{asparaenum}[(a)]

\item Since $x_e$ is a unit on $O_\E$, $\psi_W$ and $\psi_{W\setminus e}$ agree on $O_\E$ in both cases by Lemma~\ref{11}.\eqref{11a}.
Thus, $Y_W^\circ\cong Y_{W\setminus e}^\circ\times\GG_m$ and hence the claim.

\item The hypotheses and claims in the two cases are exchanged under duality.
In view of Lemma~\ref{62} we shall only prove the first claim.

The automorphism of $\KK^\E$ defined by the assignment $x_f\mapsto x_e+x_f$ and $x_g\mapsto x_g$ for all $g\in\E\setminus\set{f}$ followed by the projection along the $e$-coordinate, induces a $(\GG_m\setminus\set{1})$-fibration
\[
\varphi\colon O_\E\setminus V(x_e+x_f)\to O_{\E\setminus e},
\]
whose fiber has Grothendieck class $\TT-1$.
By Lemmas~\ref{11}.\eqref{11c} and \ref{64},
\begin{align*}
\varphi^{-1}(X_{W\setminus e}^\circ)&=X_W^\circ\setminus V(x_e+x_f)\\
X_W\cap V(x_e+x_f)&=V(\psi_W,x_e+x_f)\\
&=V(\psi_{W\setminus e}\vert_{x_f=0},x_e+x_f)\\
&=V(\psi_{W\setminus\set{e,f}},x_e+x_f).
\end{align*}
Intersecting with $O_\E$ leads to an isomorphism
\begin{align*}
X_W^\circ\cap V(x_e+x_f)&\to X_{W\setminus\set{e,f}}^\circ\times \GG_m\\
(x':x_e:x_f)&\mapsto(x',x_e/\alpha_g)\\
(x':t\alpha_g:-t\alpha_g)&\mapsfrom(x',t)
\end{align*}
where $x':=x_{\E\setminus\set{e,f}}$ and $g\in\E\setminus\set{e,f}$ is fixed. 
It follows that
\[
[X_W^\circ]=[X_{W\setminus e}^\circ]\cdot(\TT-1)+[X_{W\setminus\set{e,f}}^\circ]\cdot\TT
\]
and hence the claim.

\item Write $\psi_W=\psi_{W_1}\cdot\psi_{W_2}$ for some partition $\E=\E_1\sqcup\E_2$ as in Lemma~\ref{11}.\eqref{11b}.
By part~\eqref{52a}, we may assume that $\abs{\E_i}\ge2$ for $i=1,2$.
Then there is an isomorphism
\begin{align*}
Y_W^\circ&\to Y_{W_1}^\circ\times Y_{W_2}^\circ\times\GG_m\\
(x_{\E_1}:x_{\E_2})&\mapsto(x_{\E_1},x_{\E_2},x_{e_1}/x_{e_2})\\
(tx_{\E_1}/x_{e_1}:x_{\E_2}/x_{e_2})&\mapsfrom(x_{\E_1},x_{\E_2},t)
\end{align*}
where $e_i\in\E_i$ is fixed for $i=1,2$.
It follows that $[Y_W^\circ]=[Y_{W_1}^\circ]\cdot[Y_{W_2}^\circ]\cdot\TT$ and hence the claim.\qedhere

\end{asparaenum}
\end{proof}


The projective and torus complements are related by the formula below.

\begin{prp}[Grothendieck class and toric stratification]\label{56}
Suppose that $\M=\M_W$ has rank $\rk\M>0$.
\begin{enumerate}[(a)]
\item\label{56a} Then 
\[
[Y_W]=\sum_{S\subseteq\E=\cl_\M(S)}[Y^\circ_{W\vert_S}]\in\K_0(\Var_\KK).
\]
\item\label{56b} In particular,
\[
[Y_W]\equiv\sum_{\substack{S\subseteq\E=\cl_\M(S)\\
\text{$\M\vert_S$ connected}}}[Y^\circ_{W\vert_S}]\mod\TT.
\]
\item\label{56c} If $\M$ is loopless, then 
\[
[Y^\circ_W]\equiv\sum_{\substack{S\subseteq\E=\cl_\M(S)\\
\text{$\M\vert_S$ connected}}} (-1)^{\abs{\E\setminus S}}[Y_{W\vert_S}]\mod\TT.
\]
\end{enumerate}
\end{prp}

\begin{proof}\
\begin{asparaenum}[(a)]

\item We study the stratification 
\[
Y_W=\bigsqcup_{\emptyset\ne S\subseteq\E}Y_W\cap O_S
\]
and the resulting identity 
\[
[Y_W]=\sum_{\emptyset\ne S\subseteq\E}[Y_W\cap O_S]\in\K_0(\Var_\KK)
\]
in the Grothendieck ring:
If $\cl_\M(S)\neq \E$, then for each $B\in\B_\M$ there is an $e\in B\setminus S$ and hence $x^B\vert_{O_S}=0$. 
In this case, $\psi_W$ vanishes identically on $O_S$ by Remark~\ref{12}.\eqref{12d}, and hence $Y_W\cap
O_S=\emptyset$.
Otherwise, $S\ne\emptyset$ by the rank hypothesis and $[Y_W\cap O_S]=[Y^\circ_{W\vert_S}]$ by Lemma~\ref{64}.
The formula in \eqref{56a} follows.

\item follows from \eqref{56a} using Lemma~\ref{52}.\eqref{52b}.

\item follows from \eqref{56a} using M\"obius inversion and Proposition~\ref{50}.\eqref{50b}.\qedhere

\end{asparaenum}
\end{proof}


As a consequence we find a formula to eliminate edges in series.
In the dual graphic case it is a result of Aluffi and Marcolli (see \cite[Prop.~5.2]{AM11}).

\begin{cor}\label{59}
If $e,f\in\E$ are in series in $\M=\M_W$ with $\rk(\M/e)>0$, $\rk(\M\setminus\set{e,f})>0$, $\cl_\M(\set{e,f})\ne\E$ and $f$ is not a coloop in $\M/e$, then in $\K_0(\Var_\KK)$
\[
[Y_W]+[Y_{W/e}]\equiv[Y_{W\setminus\set{e,f}}]\mod\TT.
\]
\end{cor}

\begin{proof}
If $e=f$ is a coloop in $\M$, then $W/e\cong W\setminus\set{e,f}$ and $[Y_W]\equiv0\mod\TT$ by Proposition~\ref{50}.\eqref{50c}.
We may thus assume that $e\ne f$, and hence $e,f$ are not coloops (see Remark~\ref{47}).

Suppose that $S\subseteq\E=\cl_\M(S)$ and hence $S^\perp\in\I_{\M^\perp}$.
By hypothesis $e^\vee,f^\vee$ are parallel in $\M^\perp$ (see Remark~\ref{47}).
If $e,f\in S$, then either $e^\vee,f^\vee$ remain parallel in $\M^\perp/S^\perp=(\M\vert_S)^\perp$, or both become loops by the strong circuit exchange axiom.
So $e,f$ are either in series or both coloops in $\M\vert_S$.
In the first case,
\begin{equation}\label{71}
[Y^\circ_{W\vert_S}]\equiv-[Y^\circ_{W\vert_S/e}]\equiv-[Y^\circ_{W/e\vert_{S\setminus\set e}}]\mod\TT
\end{equation}
by Lemma~\ref{52}.\eqref{52c}.
In the second case, $f$ is a coloop in both $\M\vert_S$ and $\M\vert_S/e$ and \eqref{71} holds trivially by Lemma~\ref{52}.\eqref{52a}.

If $e\in S\not\ni f$ and hence $e^\vee$ is parallel to $f^\vee\in S^\perp$ in $\M^\perp$, then $e^\vee$ becomes a loop in $\M^\perp/S^\perp=(\M\vert_S)^\perp$, and hence $e$ is a coloop in $\M\vert_S$.
In this case, 
\begin{equation}\label{72}
[Y^\circ_{W\vert_S}]\equiv0\mod\TT
\end{equation}
by Lemma~\ref{52}.\eqref{52a}.
Moreover, since $e^\vee,f^\vee$ are parallel in $\M^\perp$, 
\begin{align}\label{73}
\cl_\M(S)=\E\iff S^\perp\in\I_{\M^\perp}
&\implies e^\vee\not\in S^\perp\lor f^\vee\not\in S^\perp\\
&\iff e\in S\lor f\in S.\nonumber
\end{align}
Applying Proposition~\ref{56}.\eqref{56b} using \eqref{71}, \eqref{72} and \eqref{73} it follows that 
\begin{equation}\label{74}
[Y_W]\equiv-\sum_{e,f\in S\subseteq\E=\cl_\M(S)}[Y^\circ_{W/e\vert_{S\setminus\set e}}]\mod\TT.
\end{equation}
For $S=S'\sqcup\set e$, we have $\cl_\M(S)=\cl_{\M/e}(S')\sqcup\set{e}$ and hence 
\[
\E=\cl_\M(S)\iff\E\setminus\set e=\cl_{\M/e}(S').
\]
Applying Proposition~\ref{56}.\eqref{56b} to $W/e$ and using \eqref{74} it follows that
\begin{equation}\label{75}
[Y_W]+[Y_{W/e}]\equiv\sum_{f\not\in S'\subseteq\E\setminus\set{e}=\cl_{\M/e}(S')}[Y^\circ_{W/e\vert_{S'}}]\mod\TT.
\end{equation}
For $S'\subseteq\E\setminus\set{e,f}$, we have $\cl_{\M/e}(S')\setminus\set{f}=\cl_{\M/e\setminus f}(S')$.
Using that $f$ is not a coloop in $\M/e$ by hypothesis, it follows that
\begin{equation}\label{77}
\E\setminus\set{e}=\cl_{\M/e}(S')\iff\E\setminus\set{e,f}\subseteq\cl_{\M/e}(S')\iff\E\setminus\set{e,f}=\cl_{\M/e\setminus f}(S').
\end{equation}
Since $e,f$ are in series in $\M$, there are isomorphisms of configurations
\begin{equation}\label{78}
W/e\vert_{S'}\cong W/e\setminus f\vert_{S'}\cong W\setminus\set{e,f}\vert_{S'}
\end{equation}
inducing corresponding identities of matroids.
As a consequence of \eqref{77} and \eqref{78}, applying Proposition~\ref{56}.\eqref{56b} to $W\setminus\set{e,f}$ identifies the right hand side of \eqref{75} with $[Y_{W\setminus\set{e,f}}]$ as claimed.\qedhere
\end{proof}


\begin{exa}[Ears attached at an edge]\label{39}
Suppose that $G$ is a parallel connection of a simple graph with at least two edges and a cycle graph $C_n$ with $n\ge3$ edges.
By Corollary~\ref{59} and Proposition~\ref{50} then $[Y_G]\equiv0\mod\TT$ in $\K_0(\Var_\KK)$.
\end{exa}


\begin{prp}[Grothendieck class and duality]\label{57}
If $\M=\M_W$ has rank $0<\rk\M<\abs{\E}$, then
\[
[Y_{W^\perp}]=b(\M)\cdot\TT^{\nullity\M-1}+\sum_{\E\ne F\in\L_\M}b(\M\vert_F)\cdot[Y^\circ_{W/F}]\cdot\TT^{\nullity(F)}\in\K_0(\Var_\KK).
\]
In particular,
\[
[Y_{W^\perp}]\equiv\delta_{1,\nullity\M}\cdot b(\M)+\sum_{F\in\I_\M\cap\L_\M}[Y^\circ_{W/F}]\mod\TT.
\]
\end{prp}

\begin{proof}
We apply Proposition~\ref{56}.\eqref{56a} to $W^\perp$:
Using that 
\[
I\in\I_\M\iff\cl_{\M^\perp}(I^\perp)=\E^\vee,\quad W^\perp\vert_{I^\perp}\cong(W/I)^\perp,\quad Y^\circ_{(W/I)^\perp}\cong Y^\circ_{W/I}
\]
by Lemma~\ref{62}, it yields 
\[
[Y_{W^\perp}]=\sum_{I\in\I_\M}[Y^\circ_{W/I}].
\]
Setting $F:=\cl_\M(I)$ for $I\in\I_\M$, $b(\M\vert_F)$ many $I$ yield the same $F$, and $\M/I$ is obtained from $\M/F$ by adding $\abs{F\setminus I}=\nullity(F)$ many loops.
If $\I\in\B_\M$ or equivalently $F=\E$, then $[Y^\circ_{W/I}]=\TT^{\nullity\M-1}$ by Remark~\ref{12}.\eqref{12a}, otherwise $[Y^\circ_{W/I}]=[Y^\circ_{W/F}]\cdot \TT^{\nullity(F)}$ by Lemma~\ref{52}.\eqref{52a}.

For the particular claim, note that
\[
\nullity(F)=0\iff\abs{F}=\rk(F)\iff F\in\I_M\iff b(\M\vert_F)=1.\qedhere
\]
\end{proof}


\begin{cor}\label{61}
Suppose that $\M=\M_W$ satisfies $\rk\M>0$, $\nullity\M>1$, and $\abs F>1$ for all $F\in\L_\M$ of rank $\rk(F)=1$.
Then $[Y_{W^\perp}]\equiv[Y_W^\circ]\mod\TT$.
\end{cor}

\begin{proof}
By Proposition~\ref{50}.\eqref{50c} and Lemma~\ref{52}.\eqref{52a}, we may assume that $\M$ has no loops and hence $\I_\M\cap\L_\M=\set{\emptyset}$.
Then Proposition~\ref{57} yields the claim.
\end{proof}

\section{Wheels with subdivided edges}\label{95}

We start from some basic graphs that satisfy Aluffi's Conjecture~\ref{40}, and then apply our results to construct counter-examples, proving Theorem~\ref{70}.


The free matroid $U_{n,n}$ is defined by any tree $T_n$ with $n$ edges.
The matroid of the $n$-edge \emph{cycle graph} $C_n$ is the uniform matroid $U_{n-1,n}$ of rank $n-1$ with $n$ elements.
Its dual with uniform matroid $U_{1,n}$ is the \emph{banana graph} $B_n:=C_n^\perp$ consisting of parallel edges (see Figure~\ref{68}).


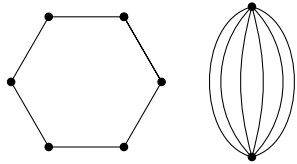
\begin{figure}[ht]
\begin{tikzpicture}[scale=1,baseline=(current bounding box.center)]
\tikzstyle{every node}=[circle,draw,inner sep=1pt,fill=black]
\draw \foreach \x in {0,...,6} {
(\x*60:1) -- (\x*60+60:1) node {} 
};
\end{tikzpicture}\quad
\begin{tikzpicture}[scale=1,baseline=(current bounding box.center)]
\tikzstyle{every node}=[circle,draw,inner sep=1pt,fill=black]
\draw (0,0) node {};
\draw (0,2) node {};
\draw \foreach \x in {0,...,5} {
(0,0) to[in=-30*\x-15,out=30*\x+15] (0,2)
};
\end{tikzpicture}
\caption{The cycle and banana graphs $C_n$ and $B_n$ for $n=6$.}\label{68}
\end{figure}


\begin{exa}[Uniform matroids of (co)rank at most $1$]\label{58}
Suppose first that $\M_W=U_{n,n}$ is a free matroid.
Then $\psi_W=x^\E$ is a monomial.
For $n\ge1$, $Y_W=Y^\circ_W=O_\E$ and hence $[Y_W]=[Y^\circ_W]\equiv0\mod\TT$.
Otherwise, $Y_W=Y^\circ_W$ is a point and $[Y_W]=[Y^\circ_W]\equiv1\mod\TT$.

Consider now a rank $1$ uniform matroid $\M_W=U_{1,n}$.
Then $[Y_W]\equiv1\mod\TT$ by Proposition~\ref{50}.\eqref{50a} and the above.
By Lemma~\ref{52}.\eqref{52c}, it suffices to compute $[Y_W^\circ]\mod\TT$ for $n=2$, where $\psi_W=x_e-x_f$ and hence $Y_W^\circ=O_{\set{e,f}}\setminus\set1$ with $[Y_W^\circ]\equiv-1\mod\TT$.

Finally, consider a corank $1$ uniform matroid $\M_W=U_{n-1,n}$.
By Lemma~\ref{52}.\eqref{52c} and Corollary~\ref{59} it suffices to consider the case where $n=3$, where $[Y_W]\equiv[Y_W^\circ]\equiv[Y_{W^\perp}^\circ]\equiv1\mod\TT$ by Proposition~\ref{56}.\eqref{56b} and Lemma~\ref{62} since $U_{2,3}^\perp=U_{1,3}$.
\end{exa}


However, our results lead to the following counter-example to Aluffi's Conjecture~\ref{40}.
The failure comes from the presence of edges in series.


\begin{exa}[$3$-wheel with subdivided edges]\label{63}
We apply Proposition~\ref{56}.\eqref{56c} to the complete graph $K_4$ on $4$ vertices.
The sum runs over all $2$-connected subgraphs of $K_4$ with four vertices.
Deleting any of the $6$ edges yields a graph $G$, deleting any of the $3$ pairs of non-adjacent edges yields a cycle graph $C_4$. 
Theorem~\ref{10} applies to $K_4$ and $G$.
Using Example~\ref{58} we obtain
\begin{equation}\label{65}
[Y^\circ_{K_4}]\equiv[Y^{}_{K_4}]-6\cdot[Y^{}_G]+3\cdot[Y^{}_{C_4}]\equiv 0-0+3\cdot(-1)\equiv-3\mod\TT.
\end{equation}
Let $\wh K_4$ and $\wh K_4^\perp$ denote the graphs obtained from $K_4$ by replacing each edge with two parallel edges or subdividing it into two edges, respectively.
Since $K_4$ is self-dual, $\wh K_4$ and $\wh K_4^\perp$ are mutually dual.
By Corollary~\ref{61}, Lemma~\ref{52}.\eqref{52c} and \eqref{65}, then
\[
[Y^{}_{H^\perp}]\equiv[Y^\circ_H]\equiv[Y^\circ_{K_4}]\equiv-3\mod\TT.
\]
\end{exa}


The basic idea of Example~\ref{63} applied to wheel graphs yields counter-examples with arbitrary large Euler characteristic.


\begin{exa}[Wheels with subdivided edges]\label{49}
Let $n\ge3$ and consider the graph $\wh W_n$ obtained from the wheel $W_n$ by subdividing each edge into two edges (see Figure~\ref{82}).

\begin{figure}[ht]
\begin{tikzpicture}[scale=1,baseline=(current bounding box.center)]
\tikzstyle{every node}=[circle,draw,inner sep=1pt,fill=black]
\draw (0,0) node {};
\draw \foreach \x in {0,...,11} {
(0,0) -- (\x*30:2) -- (\x*30+30:2) node {} 
};
\end{tikzpicture}\qquad
\begin{tikzpicture}[scale=1,baseline=(current bounding box.center)]
\tikzstyle{every node}=[circle,draw,inner sep=1pt,fill=black]
\tikzstyle{empty}=[draw=none,fill=none]
\draw (0,0) node {};
\draw \foreach \x in {0,...,11} {
(0,0) -- (\x*30:2) node [midway] {} -- (\x*30+30:2) node [midway] {} node {} 
};
\draw (0:0) to node [empty,above right] {$e$} (0:1);
\draw (0:1) to node [empty,above] {$f$} (0:2);
\end{tikzpicture}
\caption{The wheel graph $W_n$ and the graph $\wh W_n$ for $n=12$.}\label{82}
\end{figure}
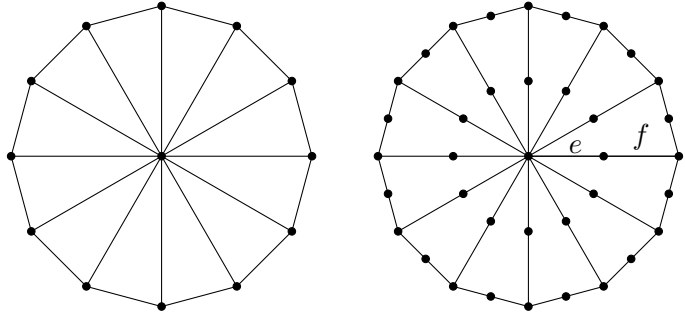

By Proposition~\ref{56}.\eqref{56b} and Lemma~\ref{52}.\eqref{52c}, 
\[
[Y_{\wh W_n}]\equiv[Y_{\wh W_n}^\circ]\equiv[Y_{\vphantom{\wh W_n}\smash{W_n}}^\circ]\mod\TT.
\]
In order to compute the latter, we apply Proposition~\ref{56}.\eqref{56c}.
To this end consider $S\subseteq\E=\cl(S)$ such that $\M\vert_S$ is connected, and call $S$ \emph{redundant} if $[Y_{W_n\vert_S}]\equiv0\mod\TT$.
In particular, $S$ must contain at least two spokes.
If however $S$ contains all spokes, then the central vertex of $W_n\vert_S$ is a fat nexus, and $S$ is redundant by Theorem~\ref{10}. 

Suppose first that $S$ contains at least $3$ spokes.
Then Corollary~\ref{59} applies to successively contract in $W_n$ all series of rim edges between neighboring spokes in $S$, using Proposition~\ref{50}.\eqref{50c} to drop redundant sets $S$ containing coloops.
This makes $S$ a set as considered for a wheel graph $W_m$ of smaller size $3\le m<n$.
The preceding argument shows that $S$ is redundant.

It remains to consider the irredundant sets $S$ containing exactly two spokes (in series), which come in three types (see Figure~\ref{69}).

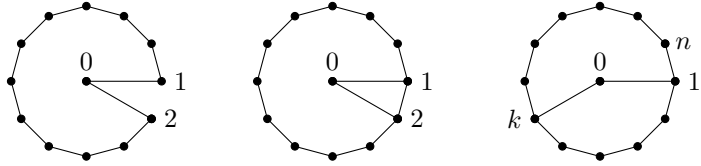
\begin{figure}[ht]
\begin{tikzpicture}[scale=1,baseline=(current bounding box.center)]
\tikzstyle{every node}=[circle,draw,inner sep=1pt,fill=black]
\draw (0,0) node [label=above:\small$0$] {} -- (0:1) node [label=right:\small$1$] {};
\draw (0,0) node {} -- (-30:1) node {};
\draw (-30:1) node [label=right:\small$2$] {};
\draw \foreach \x in {0,...,10} {
(\x*30:1) -- (\x*30+30:1) node {} 
};
\end{tikzpicture}\qquad
\begin{tikzpicture}[scale=1,baseline=(current bounding box.center)]
\tikzstyle{every node}=[circle,draw,inner sep=1pt,fill=black]
\draw (0,0) node [label=above:\small$0$] {} -- (0:1) node [label=right:\small$1$] {};
\draw (0,0) node {} -- (-30:1) node {};
\draw (-30:1) node [label=right:\small$2$] {};
\draw \foreach \x in {0,...,11} {
(\x*30:1) -- (\x*30+30:1) node {} 
};
\end{tikzpicture}\qquad
\begin{tikzpicture}[scale=1,baseline=(current bounding box.center)]
\tikzstyle{every node}=[circle,draw,inner sep=1pt,fill=black]
\draw (0,0) node [label=above:\small$0$] {} -- (0:1) node [label=right:\small$1$] {};
\draw (0,0) node {} -- (-150:1) node [label=left:\small$k$] {};
\draw (30:1) node [label=right:\small$n$] {};
\draw \foreach \x in {0,...,11} {
(\x*30:1) -- (\x*30+30:1) node {} 
};
\end{tikzpicture}
\caption{Sets $S$ containing exactly two spokes for $n=12$.}\label{69}
\end{figure}

For the first type $W_n\vert_S$ is the cycle graph $C_{n+1}$.
By symmetry it occurs $n$ times.
By Example~\ref{58}, each occurrence contributes
\[
(-1)^{\abs{\E\setminus S}}\cdot[Y_{W_n\vert_S}]\equiv (-1)^{2n-(n+1)}\cdot(-1)^{n}\equiv-1\mod\TT
\]
The second type has no contribution as can be seen by applying Corollary~\ref{59} and Proposition~\ref{50}.\eqref{50a} to the two spokes.

Suppose now that $S$ is of the third type with spokes in $S$ connected to vertex $1$ and $3\le k\le n-1$ on the rim of the wheel graph $W_n$ (see Figure~\ref{69}).
By symmetry this case occurs ${n\choose 2}-n$ times.
Applying Corollary~\ref{59} and Proposition~\ref{50}.\eqref{50c} successively to the rim edges in series as before, reduces to the case $n=4$ and $k=3$.
The total sign of this reduction equals
\[
(-1)^{k-3}\cdot(-1)^{n-k-1}=(-1)^n.
\]
Now applying the preceding argument to the two spokes in $W_n\vert_S$ results in a square with diagonal, which is redundant by Theorem~\ref{10} and a cycle graph $C_4$ which contributes $-1$ by Example~\ref{58}.
Thus, the contribution of each $S$ of the third type equals
\[
(-1)^{\abs{\E\setminus S}}\cdot[Y_{W_n\vert_S}]\equiv
(-1)^{n-2}\cdot(-1)^n\cdot(-1)\equiv-1\mod\TT.
\]
To summarize,
\[
[Y^{}_{\wh W_n}]\equiv[Y_{\vphantom{\wh W_n}\smash{W_n}}^\circ]\equiv-{n\choose 2}\mod\TT.
\]
\end{exa}


A slight modification of Example~\ref{49} serves to prove Theorem~\ref{70}.


\begin{exa}[Wheels with all edges but one spoke subdivided]\label{60}
Consider the graph $\wh W_n/f$ obtained from the wheel graph $W_n$ by subdividing all edges except for one spoke into two edges (see~Figures~\ref{82} and \ref{81}).

\begin{figure}[ht]
\begin{tikzpicture}[scale=1,baseline=(current bounding box.center)]
\tikzstyle{every node}=[circle,draw,inner sep=1pt,fill=black]
\draw (0,0) node {};
\draw \foreach \x in {1,...,11} {
(0,0) -- (\x*30:2) node [midway] {} -- (\x*30+30:2) node [midway] {} node {} 
};
\draw (0,0) -- (0:2) -- (30:2) node [midway] {} node {};
\end{tikzpicture}\qquad
\begin{tikzpicture}[scale=1,baseline=(current bounding box.center)]
\tikzstyle{every node}=[circle,draw,inner sep=1pt,fill=black]
\draw (0,0) node {};
\draw \foreach \x in {1,...,11} {
(0,0) -- (\x*30:2) node [midway] {} -- (\x*30+30:2) node [midway] {} node {} 
};
\draw (0:2) -- (30:2) node [midway] {} node {};
\end{tikzpicture}
\caption{The graphs $\wh W_n/f$ and $\wh W_n\setminus\set{e,f}$ for $n=12$.}\label{81}
\end{figure}
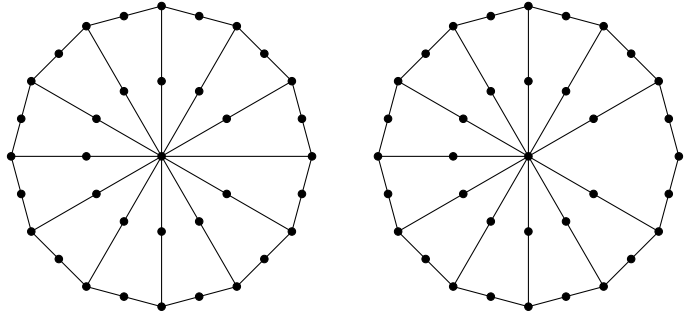

The sum in the formula in Proposition~\ref{56}.\eqref{56b} runs over $S\in\set{\E,\E\setminus\set e}$ where $e$ is the simple spoke.
Applying Corollary~\ref{59} and Proposition~\ref{50}.\eqref{50c} successively to contract series of edges as in Example~\ref{49}, yields
\begin{align*}
[Y^{}_{\wh W_n/f}]&\equiv[Y_{\wh W_n\setminus\set{e,f}}^\circ]+[Y_{\wh W_n/f}^\circ]\\
&\equiv(-1)^{2n}[Y_{W_{n-1}}^\circ]+(-1)^{2n-1}[Y_{W_n}^\circ]\\
&\equiv{n\choose 2}-{n-1\choose2}\equiv n-1\mod\TT
\end{align*}
by Example~\ref{49}.
The corresponding negative value $-n+1$ is obtained by dividing an edge of $\wh W_n/f$ different from $e$ into two.
This covers all integers $m$ with $\abs m\ge2$.
\end{exa}

\section{Uniform matroids}\label{96}

We investigate (non-graphic) configurations with uniform matroid of (co)rank $2$ and show that the statement of Aluffi's Conjecture~\ref{40} fails.


\begin{lem}\label{83}
If $\M_W$ is connected and $\rk\M_W=2$, then $[Y_W]=\LL^{\abs{\E}-2}$.
\end{lem}

\begin{proof}
Write $W=\ideal{w^1,w^2}$ as the span of linearly independent vectors $w^1,w^2\in\KK^\E$.
With $\E$ suitably ordered, the first two entries of these vectors are $(1,0)$ and $(0,1)$, respectively.
Then $w^1\star w^2$ has first two entries $(0,0)$, but is non-zero since $\M_W$ is connected.
It follows that (see Notation~\ref{14})
\[
y_1:=q(w^1\star w^1),\quad y^2:=q(w^1\star w^2),\quad y^3:=q(w^2\star w^2)
\]
are linearly independent and extend to a basis of $V^\vee$.
By Remark~\ref{12}.\eqref{12c},
\[
Q_W=\begin{pmatrix}
y_1 & y_2 \\ 
y_2 & y_3 
\end{pmatrix},\quad \psi_W=\det(Q_W)=y_1y_3-y_2^2.
\]
For $n=3$, $X_W$ is the image of $\PP^1$ under the Veronese embedding, so 
\[
[Y_W]=[\PP^2]-[\PP^1]=\LL.
\]
By Lemma~\ref{8}.\eqref{8a}, passing to $[Y_W]$ for $n\geq 3$ adds a factor of $\LL^{n-3}$.
\end{proof}


\begin{lem}\label{85}
If $\M_W=U_{2,n}$ for some $n\geq3$, then
\[
[Y^\circ_W] \equiv (-1)^{n-1}{n-1\choose 2}\mod \TT.
\]
\end{lem}

\begin{proof}
For any $S\subseteq\E$, $\M_W\vert_S$ is uniform.
Lemma~\ref{83} shows that $[Y_{W\mid S}]\equiv1\mod \TT$ provided $\abs{S}\geq 3$, which holds if $\cl(S)=\E$ and $\M_W\vert_S$ is connected.
Proposition~\ref{56}.\eqref{56c} thus yields
\begin{align*}
[Y^\circ_W]&\equiv\sum_{k\geq 3}(-1)^{n-k}{n\choose k}\mod \TT\\
&=\sum_{k\le n-3}(-1)^k{n\choose k}
=(-1)^{n-3}{n-1\choose n-3}
=(-1)^{n-1}{n-1\choose 2}.\qedhere
\end{align*}
\end{proof}


\begin{prp}\label{84}
If $\M_W=U_{n-2,n}$ for some $n\geq4$, then
\[
[Y_W]\equiv(-1)^{n-1}\frac{n^2-n+2}2\mod\TT.
\]
\end{prp}

\begin{proof}
Write $u_{k,n}$ for $[Y^\circ_{W}]$ if $\M_W=U_{k,n}$ for some $1\leq k\leq n$.
By Proposition~\ref{56}.\eqref{56b} using Corollary~\ref{61}, Example~\ref{58} and Lemma~\ref{85},
\begin{align*}
[Y_{W}]&\equiv u_{n-2,n}+{n\choose 1}u_{n-2,n-1}+{n\choose 2}
u_{n-2,n-2}\mod \TT\\
& \equiv u_{2,n}+{n\choose 1}u_{1,n-1}\mod \TT\\
&\equiv(-1)^{n-1}{n-1\choose 2}+(-1)^n{n\choose 1}\mod \TT\\
&\equiv(-1)^{n-1}\frac{n^2-5n+2}2\mod \TT.\qedhere
\end{align*}
\end{proof}

\section*{Conclusion}

We showed that projective graph hypersurface complements $Y_G$ with a non-trivial torus action are easily constructed;
on the other hand, the Euler characteristic of such spaces can be any integer. 
Similar to the work of Belkale and Brosnan, these results seem to support the heuristic that the topology of such hypersurface complements is highly non-trivial in general, yet also tractible in many special cases.

It would be interesting to know of a full combinatorial characterization of graphs for which $Y_G$ admits a non-trivial torus action. 
This however appears to be a much more difficult problem.
It would also be interesting to know if another invariant better detects the special nature of these varieties: 
for example, does the intersection homology Euler characteristic also take on
infinitely many values?

\appendix

\section{Rules}\label{101}

We collect here the computational rules we established. 
We start with the three general identities from Lemma~\ref{62} and Proposition~\ref{56}

\begin{align*}
[Y^\circ_{W^\perp}]\equiv[Y^\circ_W]&\equiv\sum_{\substack{S\subseteq\E=\cl(S)\\\M_W\vert_S\text{ connected}}}(-1)^{\abs{\E\setminus S}}[Y_{W\vert_S}]\mod\TT,\\
[Y_W]&\equiv\sum_{\substack{S\subseteq\E=\cl(S)\\\M_W\vert_S\textrm{ connected}}}
[Y^\circ_{W\vert_S}]\mod\TT.
\end{align*}

Table~\ref{88} gives an overview of rules that follow from special elements or properties of the matroid.
The entries of the middle and right columns describe the class in $\K_0(\Var_\KK)$ modulo $\TT$ of respectively $Y^\circ_W$ and $Y_W$ if the matroid $\M_W$ exhibits the feature described in the left column.
We suppress the detailed hypotheses needed for trivial examples and refer to Propositions~\ref{50}, Lemma~\ref{52} and Corollary~\ref{59} instead.

\begin{table}[ht]
\caption{Matroid specific identities in $\K_0(\Var_\KK)\mod\TT$}\label{88}
\begin{tabular}{lcc}
\toprule
feature of $\M_W$ & $[Y_W^\circ]\mod\TT$ & $[Y_W]\mod\TT$\\
\midrule
$e\in\E$ loop & $0$ & $[Y_{W\setminus e}]$ \\
\midrule
$e,f\in\E$ parallel & $-[Y_{W\setminus e}^\circ]$ & $[Y_{W\setminus e}]$ \\
\midrule
$e\in\E$ coloop & $0$ & $0$ \\
\midrule
$e,f\in\E$ in series & $-[Y_{W/e}^\circ]$ & $-[Y_{W/e}]+[Y_{W\setminus\set{e,f}}]$ \\
\midrule
$\M_W$ (loopless) disconnected & 0 & 0\\
\bottomrule
\end{tabular}
\end{table}

\section{Examples}\label{102}

Table~\ref{87} gives an overview of examples we computed.
Recall that $T_n$ is any tree with $n$ edges, $B_n$ and $C_n$ are the banana and cycle graphs with $n$ edges (see Figure~\ref{68}), $W_n$ is a wheel with $n$ spokes, $\wh W_n$ obtained from $W_n$ by dividing all edges, and $f$ a spoke edge in $\wh W_n$ (see Figure~\ref{82}).


\newcounter{tabrow}[table]
\renewcommand{\thetabrow}{\arabic{tabrow}}
\newcolumntype{N}{>{\refstepcounter{tabrow}\thetabrow}c}
\AtBeginEnvironment{tabular}{\setcounter{tabrow}{0}}


\begin{table}[ht]
\caption{Overview of examples}\label{87}
\begin{tabular}{@{}Nccccc@{}}
\toprule
\multicolumn{1}{c}{$\#$} & $G$ & $\rk\M_G$ & $\abs{\E}$ & $[Y_G^\circ]\mod\TT$ & $[Y_G]\mod\TT$
\\

\midrule
& $T_n$ & $n$ & $n$ & $\delta_{1,n}$ & $\delta_{1,n}$
\\

\midrule
& $C_n$ & $n-1$ & $n$ & $(-1)^{n-1}$ & $(-1)^{n-1}$
\\

\cmidrule{2-6}
& $B_n$ & $1$ & $n$ & $(-1)^{n-1}$ & $1$
\\

\midrule
& $W_n$ & $n$ & $2n$ & $-{n\choose 2}$ & $0$
\\

\midrule
& $\wh W_n$ & $3n$ & $4n$ & $-{n\choose 2}$ & $-{n\choose 2}$ 
\\

\midrule
& $\wh W_n/f$ & $3n-1$ & $4n-1$ & $n\choose 2$ & $n-1$
\\

\midrule
\label{87b} &
\begin{tikzpicture}[scale=0.5,baseline=(current bounding box.center)]
\tikzstyle{every node}=[circle,draw,inner sep=1pt,fill=black]
\draw (1,1) -- (3,1);
\draw (0,2) node {} -- (0,0) node {};
\draw \foreach \x in {1,...,3} {(0,0) -- (\x,1) node {} -- (0,2)};
\end{tikzpicture}
& $4$ & $9$ & $10$ & $0$
\\

\cmidrule{2-6}
\label{87a} &
\begin{tikzpicture}[scale=0.5,baseline=(current bounding box.center)]
\tikzstyle{every node}=[circle,draw,inner sep=1pt,fill=black]
\draw (0,2) node {} -- (3,2) node {} -- (3,0) node {} -- (0,0) node {} -- cycle;
\draw (1,1) node {} -- (2,1) node {};
\draw (0,0) -- (1,1) -- (0,2);
\draw (3,0) -- (2,1) -- (3,2);
\end{tikzpicture}
& $5$ & $9$ & $10$ & $1$
\\

\midrule
\label{87c} & 
\begin{tikzpicture}[scale=0.5,baseline=(current bounding box.center)]
\tikzstyle{every node}=[circle,draw,inner sep=1pt,fill=black]
\draw (0,2) node {} -- (3,2) node {} -- (3,0) node {} -- (0,0) node {} -- cycle;
\draw (1,1) node {} -- (2,1) node {};
\draw (2,1) -- (0,0) -- (1,1) -- (0,2);
\draw (3,0) -- (2,1) -- (3,2);
\end{tikzpicture}
& $5$ & $10$ & $-15$ & $0$
\\

\midrule
\label{87d} &
\begin{tikzpicture}[scale=0.5,baseline=(current bounding box.center)]
\tikzstyle{every node}=[circle,draw,inner sep=1pt,fill=black]
\draw (0,2) node {};
\draw (0,0) node {};
\draw (0,1) -- (3,1);
\draw \foreach \x in {0,...,3} {
(0,0) -- (\x,1) node {} -- (0,2)
};
\end{tikzpicture}
& $5$ & $11$ & $28$ & $0$
\\

\cmidrule{2-6}
\label{87e} &
\begin{tikzpicture}[scale=0.5,baseline=(current bounding box.center)]
\tikzstyle{every node}=[circle,draw,inner sep=1pt,fill=black]
\draw (0,3) node {} -- (4,3) node {} -- (4,0) node {} -- (0,0) node {} -- cycle;
\draw (0,3) -- (2,1) node {} -- (4,3);
\draw (1,2) node {} -- (3,2) node {};
\draw (0,0) -- (2,1) -- (4,0);
\end{tikzpicture}
& $6$ & $11$ & $28$ & $1$
\\

\midrule
\label{87f} &
\begin{tikzpicture}[scale=0.5,baseline=(current bounding box.center)]
\tikzstyle{every node}=[circle,draw,inner sep=1pt,fill=black]
\draw (1,1) -- (4,1);
\draw (0,2) node {} -- (0,0) node {};
\draw \foreach \x in {1,...,4} {(0,0) -- (\x,1) node {} -- (0,2)};
\end{tikzpicture}
& $5$ & $12$ & $-36$ & $0$
\\

\cmidrule{2-6}
\label{87g} &
\begin{tikzpicture}[scale=0.5,baseline=(current bounding box.center)]
\tikzstyle{every node}=[circle,draw,inner sep=1pt,fill=black]
\draw (0,2) node {} -- (1,1) node {} -- (2,1) node {} -- (3,1) node {} -- (4,2) node {} -- cycle;
\draw (0,2) node {} -- (1,0) node {} -- (2,0) node {} -- (3,0) node {} -- (4,2) node {};
\draw (1,0) -- (1,1);
\draw (2,0) -- (2,1);
\draw (3,0) -- (3,1);
\end{tikzpicture}
& $7$ & $12$ & $-36$ & $-2$
\\

\midrule
& $K_{3,3}$ & $5$ & $9$ & $16$ & $1$
\\

\midrule
& octahedron & $5$ & $12$ & ? & $-1$
\\

\bottomrule
\end{tabular}
\end{table}

The results for $K_{3,3}$ and the octahedron were computed using the procedures by Martin Helmer (see \cite{Hel16}) in \texttt{Macaulay2} (see \cite{M2}).

\begin{arxiv}
\section{Implementation}\label{103}

We implemented our formulas for computing $[Y_G]\mod\TT$ and $[Y_G^\circ]\mod\TT$ in $\K_0(\Var_\KK)$ in Python, using the package NetworkX.
Planar graph duality, which is crucial for our approach, needs to be performed manually.

\begin{lstlisting}[breaklines]
import euluffi as ei

# prism over a triange
E=[(1,2),(1,4),(1,5),(2,3),(2,6),(3,5),(3,6),(4,5),(4,6)]
# two tetrahedra glued along a facette (dual of E)
D=[(1,2),(2,3),(3,1),(1,4),(2,4),(3,4),(1,5),(2,5),(3,5)]

# compute toric complement for D, check it works
ei.toric_comp(D) # 10
# compute toric complement for D, store for E
ei.toric_comp_store(D,E)
# list of stored results
ei.toric_comp_storage
# compute projective complement for E
ei.proj_comp(E) # 1

# prism over a triange with square facette divided into two squares
E=[(1,2),(2,3),(3,6),(5,6),(4,5),(1,4),(2,5),(7,1),(7,4),(7,8),
   (8,3),(8,6)]
# dual of the E
D=[(1,2),(2,3),(3,4),(5,6),(5,1),(5,2),(5,3),(5,4),(6,1),(6,2),
   (6,3),(6,4)]

# compute toric complement for D, check it works
ei.toric_comp(D) # -36
# compute toric complement for D, store for E
ei.toric_comp_store(D,E)
# list of stored results
ei.toric_comp_storage
# compute projective complement for E
ei.proj_comp(E) # -2

\end{lstlisting}

\begin{lstlisting}[breaklines]
import networkx as nx

toric_comp_storage=[]
proj_comp_storage=[]

def toric_comp_lookup(L):
    if type(L)==list:
        G=nx.Graph()
        G.add_edges_from(L)
    else:
        G=L
    for tc in toric_comp_storage:
        if nx.is_isomorphic(nx.from_edgelist(tc[0]),G):
            return tc[1]
    return None

def proj_comp_lookup(L):
    if type(L)==list:
        G=nx.Graph()
        G.add_edges_from(L)
    else:
        G=L
    for pc in proj_comp_storage:
        if nx.is_isomorphic(nx.from_edgelist(pc[0]),G):
            return pc[1]
    return None

def is_fat_nexus(L,m):
    if type(L)==list:
        G=nx.Graph()
        G.add_edges_from(L)
    else:
        G=L
    F=nx.Graph()
    F.add_nodes_from(G.nodes())
    F.remove_node(m)
    N=G.neighbors(m)+[m]
    if len(N)==nx.number_of_nodes(G):
        if len(N)<4:
            return False
        else:
            return True
    F.add_edges_from([e for e in G.edges() if not (e[0] in N and e[1] in N)])
    return not nx.is_connected(F)

def has_fat_nexus(L):
    if type(L)==list:
        G=nx.Graph()
        G.add_edges_from(L)
    else:
        G=L
    for m in G.nodes():
        if is_fat_nexus(G,m):
            return True
    return False

def toric_comp_rules(L,i=0):
    if type(L)==list:
        G=nx.Graph()
        G.add_edges_from(L)
    else:
        G=L
    if nx.node_connectivity(G)<2:
        return 0
    C=nx.cycle_basis(G)
    if len(C)==1 and len(C[0])==nx.number_of_edges(G):
        return len(C[0])%2*2-1
    for m in G.nodes():
        if G.degree(m)==2:
            n=G.neighbors(m)
            F=nx.Graph()
            F.add_nodes_from(G.nodes())
            F.add_edges_from(G.edges())
            F.remove_node(m)
            c=1
            if not ((n[0],n[1]) in G.edges() or (n[1],n[0]) in G.edges()):
                F.add_edge(n[0],n[1])
                c=-1
            return c*toric_comp_rules(F,0)
    c=toric_comp_lookup(G)
    if c != None:
        return c
    if i<2:
        return toric_comp_sum(G,i+1)
    print("MISSING TORIC COMPLEMENT: "+str(G.edges()))
    return 0

def proj_comp_rules(L,i=0):
    if type(L)==list:
        G=nx.Graph()
        G.add_edges_from(L)
    else:
        G=L
    if nx.node_connectivity(G)<2:
        return 0
    C=nx.cycle_basis(G)
    if len(C)==1 and len(C[0])==nx.number_of_edges(G):
        return len(C[0])%2*2-1
    if has_fat_nexus(G):
        return 0
    for m in G.nodes():
        if G.degree(m)==2:
            n=G.neighbors(m)
            if (n[0],n[1]) in G.edges() or (n[1],n[0]) in G.edges():
                return 0
            else:
                F=nx.Graph()
                F.add_nodes_from(G.nodes())
                F.add_edges_from(G.edges())
                F.remove_node(m)
                c=proj_comp_rules(F,0)
                F.add_edge(n[0],n[1])
                return c-proj_comp_rules(F,0)
    c=proj_comp_lookup(G)
    if c != None:
        return c
    if i<2:
        return proj_comp_sum(G,i+1)
    print("MISSING PROJECTIVE COMPLEMENT: "+str(G.edges()))
    return 0

def toric_comp_sum(L,i=0,E=[]):
    if type(L)==list:
        G=nx.Graph()
        G.add_edges_from(L)
    else:
        G=L
    c=proj_comp_rules(G,i)
    if c!=0:
        print(G.edges())
        print(c)
    for e in G.edges():
        if e not in E and G.degree(e[0])>1 and G.degree(e[1])>1:
            G.remove_edge(*e)
            E=E+[e]
            c=c-toric_comp_sum(G,0,E)
            G.add_edge(*e)
    return c

def proj_comp_sum(L,i=0,E=[]):
    if type(L)==list:
        G=nx.Graph()
        G.add_edges_from(L)
    else:
        G=L
    c=toric_comp_rules(G,i)
    if c!=0:
        print(G.edges())
        print(c)
    for e in G.edges():
        if e not in E and G.degree(e[0])>1 and G.degree(e[1])>1:
            G.remove_edge(*e)
            E=E+[e]
            c=c+proj_comp_sum(G,0,E)
            G.add_edge(*e)
    return c

def toric_comp(E):
    return toric_comp_rules(E)

def proj_comp(E):
    return proj_comp_rules(E)

def toric_comp_store(E,D=None):
    global toric_comp_storage
    if D==None:
        D=E
    c=toric_comp(E)
    toric_comp_storage=toric_comp_storage+[[D,c]]

def proj_comp_store(E):
    global proj_comp_storage
    c=proj_comp(E)
    proj_comp_storage=proj_comp_storage+[[E,c]]

\end{lstlisting}
\end{arxiv}

\printbibliography
\end{document}